\documentclass[a4paper,10pt]{scrartcl}
\usepackage[utf8]{inputenc}
\usepackage[english]{babel}
\usepackage[fixlanguage]{babelbib}
\usepackage{cite}
\usepackage{amssymb, amsmath, amstext, amsopn, amsthm, amscd, amsxtra, amsfonts}
\usepackage{yfonts, mathrsfs}
\usepackage{paralist}

\usepackage{enumerate}

\usepackage[all]{xy}

\usepackage[T1]{fontenc}
\usepackage{tikz}
\usetikzlibrary{shapes,arrows,positioning}

\usepackage{ifthen}
\usepackage{colonequals}
\usepackage{array}

\usepackage[bookmarks,bookmarksnumbered,plainpages=false]{hyperref}

\swapnumbers
\newtheoremstyle{standard}
 {16pt}  
 {16pt}  
 {}  
 {}  
 {\bfseries}
 {}  
 { } 
 {{\thmname{#1~}}{\thmnumber{#2.}}\thmnote{~(#3)}} 

\newtheoremstyle{kursiv}
 {16pt}  
 {16pt}  
 {\itshape}  
 {}  
 {\bfseries}
 {}  
 { } 
 {{\thmname{#1~}}{\thmnumber{#2.}}\thmnote{~(#3)}} 

\theoremstyle{standard}
\newtheorem{defn} [subsection]{Definition}
\newtheorem{ex} [subsection]{Example}
\newtheorem{rem}   [subsection]{Remark}
\newtheorem{nota}   [subsection]{Notation}
\newtheorem{setup} [subsection]{}

\theoremstyle{definition}

 \newtheorem{problem}[subsection]{Problem}

\theoremstyle{kursiv}
\newtheorem{thm}[subsection]{Theorem}
\newtheorem{prop} [subsection]{Proposition}
\newtheorem{cor} [subsection]{Corollary}
\newtheorem{lem} [subsection]{Lemma}

\setdefaultenum{\textup (a)}{\textup i.}{\textup A.}{\textup 1.}

\newcommand{\Evol}{\mathrm{Evol}}

\newcommand{\Ad}{\mathrm{Ad}}
\newcommand{\ad}{\mathrm{ad}}
\newcommand{\evol}{\mathrm{evol}}

\newcommand{\id}{\mathrm{id}}

\newcommand{\N}{\mathbb{N}}
\newcommand{\Z}{\mathbb{Z}}

\newcommand{\R}{\mathbb{R}}
\newcommand{\K}{\mathbb{K}}
\newcommand{\C}{\mathbb{C}}

\newcommand{\g}{\mathfrak{g}}

\renewcommand{\epsilon}{\varepsilon}

\newcommand{\set}[1]{\{  #1 \}}
\newcommand{\setm}[2]{\left\{\, #1 \middle\vert #2\,\right\}}
\newcommand{\norm}[1]{\left\lVert #1 \right\rVert}
\newcommand{\abs}[1]{\left| #1 \right|}

\newcommand{\coloneq}{\colonequals}

\DeclareMathOperator{\Hom}{Hom}

\DeclareMathOperator*{\Ann}{Ann}
\newcommand{\dd}{\mathrm{d}}

\DeclareSymbolFont{bbold}{U}{bbold}{m}{n}
\DeclareSymbolFontAlphabet{\mathbbold}{bbold}
\newcommand{\Cut}[1]{\Lambda}

\newcommand{\cA}{\ensuremath{\mathcal{A}}}

\newcommand{\cC}{\ensuremath{\mathcal{C}}}

\newcommand{\cH}{\ensuremath{\mathcal{H}}}
\newcommand{\cI}{\ensuremath{\mathcal{I}}}
\newcommand{\cJ}{\ensuremath{\mathcal{J}}}

\newcommand{\cL}{\ensuremath{\mathcal{L}}}

\newcommand{\cP}{\ensuremath{\mathcal{P}}}

\newcommand{\cT}{\ensuremath{\mathcal{T}}}
\newcommand{\cU}{\ensuremath{\mathcal{U}}}

\newcommand{\Lf}{\ensuremath{\mathbf{L}}}

\newcommand{\func}[5]{#1 \colon #2 \rightarrow #3 , #4 \mapsto #5}
\newcommand{\smfunc}[3]{#1 \colon #2 \rightarrow #3}
\newcommand{\nnfunc}[4]{#1 \rightarrow #2 , #3 \mapsto #4}
\newcommand{\smset}[1]{ \left\{ #1 \right\} }

\newcommand{\seqnN}[1]{\left( #1 \right)_{n\in\N_0}}
\newcommand{\LB}[1][\cdot \hspace{1pt} , \cdot]{[\hspace{1pt} #1 \hspace{1pt} ]}

   \newcommand{\wcotimes}{\ {\widetilde\otimes}\ }
\DeclareMathOperator{\one}{\mathbf{1}}
\newcommand{\Frechet}{Fr\'echet }

\newcommand{\Alg}   {\cA}					
\newcommand{\CoAlg} {\cC}					
\newcommand{\Hopf}  {\cH}					
\newcommand{\HIdeal}{\cJ}			                

\newcommand{\lcA}{A}						
\newcommand{\lcB}{B}						

\newcommand{\lcAt}{\lcA_\otimes}				

\newcommand{\IdealA}{\cI_\lcA}					

\newcommand{\InfChar}[2]{ \g(#1 , #2) }				
\newcommand{\Char}[2]{     G(#1 , #2) }			

\newcommand{\AnnGroup}{\Ann(\HIdeal, \lcB) \cap \Char{\Hopf}{\lcB}}
\newcommand{\AnnLieAlgebra}{\Ann(\HIdeal, \lcB) \cap \InfChar{\Hopf}{\lcB}}

\newcommand{\VSCat}{\mathbf{VS}_\K}
\newcommand{\WCCat}{\mathbf{WCVS}_\K}
\newcommand{\EE}{E}
\newcommand{\FF}{F}
\newcommand{\HH}{H}
\newcommand{\VV}{\mathcal V}
\newcommand{\WW}{\mathcal W}

\newcommand{\Func}[5]{
\begin{array}{rl}
   #1 \colon #2 & \longrightarrow #3   \\
        #4 & \longmapsto     #5   \\
\end{array}
}

\newcommand{\BiFunc}[7]{
\begin{array}{rl}
   #1 \colon #2 & \longrightarrow #3   \\
        #4 & \longmapsto     #5   \\
        #6 & \longmapsto     #7   \\
\end{array}
}

\newcommand{\RT}{\ensuremath{\cT}}
\DeclareMathOperator{\OST}{OST}

\title{Character groups of Hopf algebras as infinite-dimensional Lie groups} 
 \author{G. Bogfjellmo, 
  R. Dahmen 
  and A. Schmeding
 }
 \date{}
\begin{document}

\maketitle

\begin{abstract}
In this article character groups of Hopf algebras are studied from the perspective of infinite-dimensional Lie theory.
For a graded and connected Hopf algebra we construct an infinite-dimensional Lie group structure on the character group with values in a locally convex algebra.
This structure turns the character group into a Baker--Campbell--Hausdorff--Lie group which is regular in the sense of Milnor. 
Furthermore, we show that certain subgroups associated to Hopf ideals become closed Lie subgroups of the character group.

If the Hopf algebra is not graded, its character group will in general not be a Lie group.
However, we show that for any Hopf algebra the character group with values in a weakly complete algebra is a pro-Lie group in the sense of Hofmann and Morris.
\end{abstract}

\medskip


\medskip

 (Secondary)

\tableofcontents

\section*{Introduction and statement of results} \addcontentsline{toc}{section}{Introduction and statement of results}
 Hopf algebras and their character groups appear in a variety of mathematical and physical contexts. 
 To name just a few, they arise in non-commutative geometry, renormalisation of quantum field theory (see \cite{CK98}) and numerical analysis (cf.\ \cite{Brouder-04-BIT}).
 We recommend \cite{MR2290769} as a convenient introduction to Hopf algebras and their historical development.
 
 In their seminal work \cite{MR1748177,MR1810779} Connes and Kreimer associate to the group of characters of a Hopf algebra of Feynman graphs a Lie algebra. 
 It turns out that this (infinite-dimensional) Lie algebra is an important tool to analyse the structure of the character group. 
 In fact, the character group is then called \textquotedblleft infinite-dimensional Lie group\textquotedblright \ meaning that it is associated to an infinite-dimensional Lie algebra. 
 Moreover, it is always possible to construct a Lie algebra associated to the character group of a Hopf algebra.
 These constructions are purely algebraic in nature and one may ask, whether the Lie algebras constructed in this way are connected to some kind of Lie group structure on the character group. 
 Indeed, in \cite{BS14} the character group of the Hopf algebra of rooted trees was turned into an infinite-dimensional Lie group.
 Its Lie algebra is closely related to the Lie algebra constructed in \cite{CK98} for the character group.
 \medskip
 
 These observations hint at a general theme which we explore in the present paper. 
 Our aim is to study character groups (with values in a commutative locally convex algebra) of a Hopf algebra from the perspective of infinite-dimensional Lie theory. 
 We base our investigation on a concept of $C^r$-maps between locally convex spaces known as Keller's $C^r_c$-theory\footnote{Although Keller's $C^r_c$-theory is in general inequivalent to the \textquotedblleft convenient setting\textquotedblright \ of calculus \cite{MR1471480}, in the important case of \Frechet spaces both theories coincide (e.g.\ Example \ref{ex: Butcher}).}~\cite{keller1974} (see \cite{MR830252,MR1911979,MR2261066} for streamlined expositions and Appendix \ref{app: mfd} for a quick reference). 
 In the framework of this theory, we treat infinite-dimensional Lie group structures for character groups of Hopf algebras.
 If the Hopf algebra is graded and  connected, it turns out that the character group can be made into an infinite-dimensional Lie group.
 We then investigate Lie theoretic properties of these groups and some of their subgroups. 
 In particular, the Lie algebra associated to the Lie group structure on the group of characters turns out to be the Lie algebra of infinitesimal characters.
 
 The character group of an arbitrary Hopf algebra can in general not be turned into an infinite-dimensional Lie group and we provide an explicit example for this behaviour.
 However, it turns out that the character group of an arbitrary Hopf algebra (with values in a finite dimensional algebra) is always a topological group with strong structural properties, i.e.\ it is always the projective limit of finite dimensional Lie groups. 
 Groups with these properties are accessible to Lie theoretic methods (cf.\ \cite{MR2337107}) albeit they may not admit a differential structure.   
 \medskip
 
 We now go into some more detail and explain the main results of the present paper.
 Let us recall first the definition of the character group of a Hopf algebra $(\Hopf, m_\Hopf, 1_\Hopf, \Delta_\Hopf, \epsilon_\Hopf, S_\Hopf)$ over the field $\K \in \{\R, \C\}$.
 Fix a commutative locally convex algebra $\lcB$. 
 Then the character group $\Char{\Hopf}{\lcB}$ of $\Hopf$ with values in $\lcB$ is defined as the set of all unital algebra characters
  \begin{displaymath}
   \Char{\Hopf}{\lcB} \coloneq \{\phi \in \Hom_\K (\Hopf, \lcB) \mid \phi (ab) = \phi (a)\phi (b), \forall a,b \in \Hopf \text{ and } \phi (1_\Hopf) = 1_\lcB\},
  \end{displaymath}
 where the group product is the convolution product $\phi \star \psi \coloneq m_\lcB \circ (\phi \otimes \psi ) \circ \Delta_\Hopf$.
 
 Closely related to this group is the Lie algebra of infinitesimal characters 
  \begin{displaymath}
   \InfChar{\Hopf}{\lcB} \coloneq \{ \phi \in \Hom_\K (\Hopf, \lcB) \mid \phi (ab) = \epsilon_\Hopf (a) \phi(b) + \epsilon_\Hopf(b)\phi(a)\}
  \end{displaymath}
 with the commutator Lie bracket $\LB[\phi ,\psi] \coloneq \phi \star \psi - \psi \star \phi$.
 
 It is well known that for a certain type of Hopf algebra (e.g.\ graded and connected) the exponential series induces a bijective map $\InfChar{\Hopf}{\lcB} \rightarrow \Char{\Hopf}{\lcB}$. 
 In this setting, the ambient algebra $(\Hom_\K (\Hopf, \lcB), \star)$ becomes a locally convex algebra with respect to the topology of pointwise convergence and  we obtain the following result.
 \bigskip
 
 \textbf{Theorem A} \emph{  Let $\Hopf$ be a graded and connected Hopf algebra and $\lcB$ a commutative locally convex algebra, then the group $\Char{\Hopf}{\lcB}$ of $\lcB$-valued characters of $\Hopf$ is a ($\K$-analytic) Lie group.}
  
 \emph{The Lie algebra of $\Char{H}{\lcB}$ is the Lie algebra $\InfChar{\Hopf}{\lcB}$ of infinitesimal characters.}
  \bigskip  
 
 Note that this Lie group structure recovers the Lie group structure on the character group of the Hopf algebra of rooted trees which has been constructed in \cite{BS14}. 
 For further information we refer to Example \ref{ex: Butcher}. \medskip 
 
 We then investigate the Lie theoretic properties of the character group of a graded connected Hopf algebra. 
 To understand these results first recall the notion of regularity for Lie groups.
 
 Let $G$ be a Lie group modelled on a locally convex space, with identity element $\one$, and
 $r\in \N_0\cup\{\infty\}$. We use the tangent map of the left translation
 $\lambda_g\colon G\to G$, $x\mapsto xg$ by $g\in G$ to define
 $v.g\coloneq T_{\one} \lambda_g(v) \in T_g G$ for $v\in T_{\one} (G) =: \Lf(G)$.
 Following \cite{dahmen2011} and \cite{1208.0715v3}, $G$ is called
 \emph{$C^r$-regular} if for each $C^r$-curve
 $\gamma\colon [0,1]\rightarrow \Lf(G)$ the initial value problem
 \begin{displaymath}
  \begin{cases}
   \eta'(t)&= \eta (t). \gamma(t)\\ \eta(0) &= \one
  \end{cases}
 \end{displaymath}
 has a (necessarily unique) $C^{r+1}$-solution
 $\Evol (\gamma)\coloneq\eta\colon [0,1]\rightarrow G$, and the map
 \begin{displaymath}
  \evol \colon C^r([0,1],\Lf(G))\rightarrow G,\quad \gamma\mapsto \Evol
  (\gamma)(1)
 \end{displaymath}
 is smooth. If $G$ is $C^r$-regular and $r\leq s$, then $G$ is also
 $C^s$-regular. A $C^\infty$-regular Lie group $G$ is called \emph{regular}
 \emph{(in the sense of Milnor}) -- a property first defined in \cite{MR830252}.
 Every finite dimensional Lie group is $C^0$-regular (cf.\ \cite{MR2261066}). Several
 important results in infinite-dimensional Lie theory are only available for
 regular Lie groups (see
 \cite{MR830252,MR2261066,1208.0715v3}, cf.\ also \cite{MR1471480} and the references therein).
 
 Concerning the Lie theoretic properties of the character groups our results subsume the following theorem.
 \bigskip
 
 \textbf{Theorem B} \emph{Let $\Hopf$ be a graded and connected Hopf algebra and $\lcB$ be a commutative locally convex algebra. 
   \begin{itemize}
    \item[\textup{(a)}] Then $\Char{H}{\lcB}$ is a Baker--Campbell--Hausdorff--Lie group which is exponential, i.e.~the Lie group exponential map is a global $\K$-analytic diffeomorphism.
    \item[\textup{(b)}] If $\lcB$ is sequentially complete then $\Char{\Hopf}{\lcB}$ is a $C^0$-regular Lie group. 
  \end{itemize}
}
 We then turn to a class of closed subgroups of character groups which turn out to be closed Lie subgroups. 
 For a Hopf ideal $\HIdeal$ of a Hopf algebra $\Hopf$, consider the annihilator 
 \begin{displaymath}
  \Ann (\HIdeal , \lcB) \coloneq \{ \phi \in \Hom_\K (\Hopf, \lcB) \mid \phi (a) = 0_\lcB ,\ \forall a \in \HIdeal\}.
 \end{displaymath}
 Then $\Ann (\HIdeal , \lcB) \cap \Char{\Hopf}{\lcB}$ becomes a subgroup and we obtain the following result.
 \bigskip
 
 \textbf{Theorem C} \emph{Let $\Hopf$ be a connected and graded Hopf algebra and $\lcB$ be a commutative locally convex algebra. 
  \begin{itemize}
   \item[\textup{(a)}] Then $\AnnGroup$ is a closed Lie subgroup of $\Char{\Hopf}{\lcB}$ whose Lie algebra is $\AnnLieAlgebra$.
   \item[\textup{(b)}] There is a canonical isomorphism of (topological) groups $\AnnGroup \cong \Char{\Hopf/\HIdeal}{\lcB}$, where $\Hopf/ \HIdeal$ is the quotient Hopf algebra. 
   If $\Hopf / \HIdeal$ is a connected and graded Hopf algebra (e.g.\ $\HIdeal$ is a homogeneous ideal) then this map is an isomorphism of Lie groups. 
  \end{itemize}
 }
 Note that in general $\Hopf / \HIdeal$ will not be graded and connected. 
 In these cases the isomorphism $\AnnGroup \cong \Char{\Hopf/\HIdeal}{\lcB}$ extends the construction of Lie group structures for character groups to Hopf algebras which are quotients of connected and graded Hopf algebras.
 However, this does not entail that the character groups of all Hopf algebras are infinite-dimensional Lie groups. 
 In general, the character group will only be a topological group with respect to the topology of pointwise convergence.
 We refer to Example \ref{ex: group algebra} for an explicit counter example of a character group which can not be turned into an infinite-dimensional Lie group.
 \medskip
 
 Finally, we consider a class of character groups of (non-graded) Hopf algebras whose topological group structure is accessible to Lie theoretic methods.
 The class of characters we consider are character groups with values in an algebra which is \textquotedblleft weakly complete\textquotedblright , i.e.\ the algebra is as a topological vector space isomorphic to $\K^I$ for some index set $I$. (All finite dimensional algebras are weakly complete, we refer to the Diagram \ref{setup: alg:prop} and Appendix \ref{app: weakly complete} for more information.)
 Then we obtain the following result:
 \bigskip
 
 \textbf{Theorem D} \emph{ Let $\Hopf$ be an arbitrary Hopf algebra and $\lcB$ be a commutative weakly complete algebra. 
 Then the following holds 
  \begin{itemize}
   \item[\textup{(a)}] the topological group $\Char{\Hopf}{\lcB}$ is a projective limit of finite dimensional Lie groups (a \emph{pro-Lie group} in the sense of \cite{MR2337107}).
   \end{itemize}
  A pro-Lie group is associated to a Lie algebra which we identify for $\Char{\Hopf}{\lcB}$: 
   \begin{itemize}
   \item[\textup{(b)}] the pro-Lie algebra $\cL (\Char{\Hopf}{\lcB})$ of the pro-Lie group $\Char{\Hopf}{\lcB}$ is the Lie algebra of infinitesimal characters $\InfChar{\Hopf}{\lcB}$.
  \end{itemize}
 } 
 Note that pro-Lie groups are in general only topological groups without a differentiable structure attached to them.
 However, these groups admit a Lie theory which has been developed in the extensive monograph \cite{MR2337107}.
 The results on the pro-Lie structure are somewhat complementary to the infinite-dimensional Lie group structure.
 If the Hopf algebra $\Hopf$ is graded and connected and $\lcB$ is a commutative weakly complete algebra, then the pro-Lie group structure of $\Char{\Hopf}{\lcB}$ is compatible with the infinite-dimensional Lie group structure of $\Char{\Hopf}{\lcB}$ obtained in Theorem A.
 
%
 
 \section*{Acknowledgements}
 The research on this paper was partially supported by the project \emph{Topology in Norway} (NRC project 213458) and \emph{Structure Preserving Integrators, Discrete Integrable Systems and Algebraic Combinatorics} (NRC project 231632).
 
\section{Linear maps on (connected) coalgebras}

In this preliminary section we collect first some basic results and notations used throughout the paper (also cf.\ Appendices \ref{app: mfd} - \ref{app: weakly complete}). 
Most of the results are not new, however, we state them together with a proof for the reader's convenience.

\begin{nota}
  We write $\N\coloneq \smset{1,2,3,\ldots}$, and $\N_0 \coloneq\N\cup\smset{0}$. 
  Throughout this article (with the exception of Appendix \ref{app: weakly complete}), $\K$ denotes either the field $\R$ of real or the field $\C$ of complex numbers. 
\end{nota}

\begin{setup}[Terminology]
 By the term \emph{(co-)algebra}, we always mean an (co-)associative unital $\K$-(co-)algebra. 
 The \emph{unit group} or \emph{group of units} of an algebra $\Alg$ is the group of its invertible elements and is denoted by $\Alg^\times$.

 A \emph{locally convex space} is a locally convex Hausdorff topological vector space. 
 and a \emph{weakly complete space} is a locally convex space which is topologically isomorphic to $\K^I$ for an index set $I$ (see Definition \ref{def: weakly_complete_space}). 
 A \emph{locally convex algebra (weakly complete) algebra} is a topological algebra whose underlying topological space is locally convex (weakly complete). (see also Lemma \ref{lem: fundamental_lemma_of_weakly_complete_algebras}) 
 Finally, a \emph{continuous inverse algebra} (\emph{CIA}) is a locally convex algebra with an open unit group and a continuous inversion.

 If we want to emphasize that an algebraic structure, such as a vector space or an algebra, carries no topology, we call it an \emph{abstract vector space} or \emph{abstract algebra}, etc.
\end{setup}

\begin{setup}[Algebra of linear maps on a coalgebra]		\label{setup: maps on coalgebra}
 \sloppy
 Throughout this section, let $\CoAlg=(\CoAlg,\Delta_\CoAlg,\epsilon_\CoAlg)$ denote an abstract coalgebra and let $\lcB$ denote a locally convex topological algebra, e.g.~ a Banach algebra. Then we consider the locally convex space 
 \[
  \lcA \coloneq \Hom_\K (\CoAlg, \lcB)
 \]
 of all $\K$-linear maps from $\CoAlg$ into $\lcB$. We will give this space the topology of pointwise convergence, i.e.~we embed $A$ into the product space $\lcB^\CoAlg$ with the product topology.\\
 The most interesting case is $\lcB=\K$:
 In this case, $\lcA$ is the algebraic dual of the abstract vector space $\CoAlg$ with the weak*-topology.

 The space $\lcA$ becomes a unital algebra with respect to the \emph{convolution product} (cf.\ \cite[Section IV]{MR0252485})
  \begin{displaymath}
   \star \colon \lcA \times \lcA \rightarrow \lcA, (h,g) \mapsto m_\lcB \circ (h \otimes g) \circ \Delta_\CoAlg.
  \end{displaymath}
 Here $m_\lcB \colon \lcB \otimes \lcB \rightarrow \lcB$ is the algebra multiplication. 
 The unit with respect to $\star$ is the map $1_\lcA = u_\lcB \circ \epsilon_\CoAlg$ where we defined $\func{u_\lcB}{\K}{\lcB}{z}{z 1_\lcB}$.
 We will now show that the map $\smfunc{\star}{\lcA\times\lcA}{\lcA}$ is continuous, hence turns $(A,\star)$ into a locally convex algebra:

 Since the range space $\lcA=\Hom_{\K}(\CoAlg,\lcB)$ carries the topology of pointwise convergence, we fix an element $c\in \CoAlg$. 
 We write $\Delta_\CoAlg(c)\in \CoAlg\otimes \CoAlg$ in Sweedler's sigma notation (see \cite[Notation 1.6]{MR1321145} or \cite[Section 1.2]{MR0252485}) as a finite sum
    \[
     \Delta_\CoAlg(c) = \sum_{(c)} c_{1}\otimes c_{2}.
    \]
    Then the convolution product $\phi\star\psi$ evaluated at point $c$ is of the form:
    \begin{align*}
     (\phi \star \psi) (c)	  &  	= 	m_\lcB\circ (\phi\otimes\psi)\circ \Delta_\CoAlg\left(c\right)
			  	= 		m_\lcB\circ ( \phi\otimes\psi) \left(	\sum_{(c)} c_{1}\otimes c_{2}	\right) 
			\\&	= \sum_{(c)} 	m_\lcB\left(\phi(c_{1}) \otimes \psi(c_{2}) \right)
				= \sum_{(c)} 	\phi(c_{1}) \cdot \psi(c_{2}).
    \end{align*}
    This expression is continuous in $(\phi,\psi)$ since point evaluations are continuous as well as multiplication in the locally convex algebra $\lcB$. 
    \end{setup}
    
 \begin{rem}
  Note that the multiplication of locally convex algebra $\lcB$ is assumed to be a continuous bilinear map $\lcB \times \lcB \rightarrow \lcB$.
  However, we did not need to put a topology on the space $\lcB\otimes\lcB$ nor did we say anything about the continuity of the linear map $\smfunc{m_\lcB}{\lcB\otimes\lcB}{\lcB}$.
 \end{rem}

\begin{lem}[Properties of the space $\lcA$]					\label{lem: completeness of A}
 Let $\lcA = \Hom_\K(\CoAlg,B)$ as above.
 \begin{itemize}
  \item [\textup{(a)}] As a locally convex space (without algebra structure), the space $\lcA$ is isomorphic to $\lcB^I$, where the cardinality of the index set $I$ is equal to the dimension of $\CoAlg$.
  \item [\textup{(b)}] If the vector space $\CoAlg$ is of countable dimension and $\lcB$ is a \Frechet space,  $\lcA$ is a \Frechet space as well.
  \item [\textup{(c)}] The locally convex algebra $\lcA$ is (Mackey/sequentially) complete if and only if the algebra $\lcB$ is (Mackey/sequentially) complete.
 \end{itemize}

\end{lem}
\begin{proof} 
 \begin{enumerate}
  \item[(a)] A linear map is uniquely determined by its valued on a basis $(c_i)_{i \in I}$ of $\CoAlg$. 
 \item[(b)] As a locally convex space $\lcA\cong \lcB^I$. 
 Since $I$ is countable and $\lcB$ is a \Frechet space,  $\lcA$ is a countable product of \Frechet spaces, whence a \Frechet space.
 \item[(c)] By part (a), $\lcB$ is a closed vector subspace of $\lcA$. 
 So every completeness property of $\lcA$ is inherited by $\lcB$. On the other hand, products of Mackey complete (sequentially complete, complete) spaces are again of this type. \qedhere
 \end{enumerate}
\end{proof}

The terms \emph{abstract gradings} and \emph{dense gradings} used in the next lemma are defined in \ref{setup: abstract grading} and \ref{setup: dense grading}, respectively.
 
 \begin{lem}								\label{lem: grading and CIA}
  Let $\CoAlg$ be an abstract coalgebra, let $\lcB$ be a locally convex algebra, and set $\lcA=\Hom_\K(\CoAlg,\lcB)$ as above.
  \item [\textup{(a)}]
    If $\CoAlg$ admits an (abstract) grading $\CoAlg=\bigoplus_{n=0}^\infty \CoAlg_n$, the bijection 
    \begin{equation}\label{eq: nat:grad}
     \nnfunc{\lcA = \Hom_{\K}\left(\bigoplus_{n=0}^\infty \CoAlg_n,\lcB\right) }{ \prod_{n=0}^\infty \Hom_{\K}(\CoAlg_n ,\lcB)}{\phi}{\left(\phi|_{\CoAlg_n}\right)_{n\in\N_0} }
    \end{equation}
    turns $\lcA$ into a densely graded algebra with respect to $\seqnN{\Hom_{\K}(\CoAlg_n,\lcB)}$.
  \item [\textup{(b)}]
  If in addition $\CoAlg$ is connected and $\lcB$ is a CIA, then $\lcA$ is a CIA as well.
 \end{lem}
 \begin{proof}
  \begin{enumerate}
   \item It is clear that the map \eqref{eq: nat:grad} is an isomorphism of topological vector spaces. 
   Via this dualisation, the axioms of the graded coalgebra (see \ref{setup: abstract grading} (c) directly translate to the axioms of densely graded locally convex algebra (see \ref{setup: dense grading} (b)).
  \item By \ref{lem: unit_groups_of_graded_algebras}, we know that the densely graded locally convex algebra $\lcA$ is a CIA, if $\lcA_0=\Hom_{\K}(\CoAlg_0,\lcB)$ is a CIA. Since we assume that $\CoAlg$ is connected, this means that $\CoAlg\cong\K$ and hence $\lcA_0\cong\lcB$. The assertion follows. \qedhere
  \end{enumerate}
 \end{proof}

 From Lemma \ref{lem: grading and CIA} and \ref{thm: glockner_CIA_BCH} one easily deduce the following proposition.
\newpage
 \begin{prop}[$\lcA^\times$ is a Lie group]				\label{prop: unit group Lie group}
  Let $\CoAlg$ be an abstract graded connected coalgebra and let $\lcB$ be a Mackey complete CIA. Then the unit group $\lcA^\times$ of the densely graded algebra $\lcA = (\Hom_{\K}(\CoAlg,\lcB),\star)$ is a BCH--Lie group.
  The Lie algebra of the group $\lcA^\times$ is $(\lcA,[\cdot,\cdot])$, where $[\cdot,\cdot]$ denotes the usual commutator bracket.

  Furthermore, the Lie group exponential function of $\lcA^\times$ is given by the exponential series, whence it restricts to the exponential function constructed in Lemma \ref{lem: exp_and_log}.
 \end{prop}

 \begin{thm}[Regularity of $\lcA^\times$]				\label{thm: unit_group_regular}
  Let $\CoAlg$ be an abstract graded connected coalgebra and let $\lcB$ be a Mackey complete CIA. As above, we set
  $ \lcA\coloneq(\Hom_\K(\CoAlg,\lcB),\star)$ and assume that $\lcB$ is commutative or locally m-convex (the topology is generated by a system of submultiplicative seminorms).
  \begin{itemize}
   \item [\textup{(a)}] The Lie group $\lcA^\times$ is $C^1$-regular.
   \item [\textup{(b)}] If in addition, the space $\lcB$ is sequentially complete, then $\lcA^\times$ is $C^0$-regular.
  \end{itemize}
  In both cases the associated evolution map is even $\K$-analytic.
 \end{thm}
 
 \begin{proof}
  Since $\lcB$ is a commutative CIA or locally m-convex, the algebra $\lcB$ has the (GN)-property by \ref{setup: GN_commutative_locally_m_convex}.
  The algebra $\lcA\coloneq\Hom_\K(\CoAlg,\lcB)$ is densely graded with $\lcA_0\cong \lcB$ by \ref{lem: grading and CIA}. 
  We claim that since $\lcA_0 = \lcB$ has the (GN)-property, the same holds for $\lcA$ (the details are checked in Lemma \ref{lem: densely_GN} below).
 
  By Lemma \ref{lem: completeness of A}, we know that $\lcA$ and $\lcB$ share the same completeness properties, i.e.~the algebra $\lcA$ is Mackey complete if $\lcB$ is so and the same hold for sequential completeness.
  In conclusion, the assertion follows directly from \ref{setup: GN_regular}.
 \end{proof}

 \begin{lem}								\label{lem: densely_GN}
  Let $\lcA$ be a densely graded algebra and denote the dense grading by $(A_N)_{N\in\N_0}$.
  Then $A$ has the (GN)-property (see Definition \ref{defn: GN_property}) if and only if the subalgebra $A_0$ has the (GN)-property.
 \end{lem}
 \begin{proof}
  Let us first see that the condition is necessary. 
  Pick a continuous seminorm $p_0$ on $\lcA_0$. 
  Then $P \coloneq p_0 \circ \pi_0$ is a continuous seminorm on $\lcA$ (where $\pi_0 \colon \lcA \rightarrow \lcA_0$ is the canonical projection).
  Following Definition \ref{defn: GN_property} there is a continuous seminorm $Q$ on $\lcA$ and a number $M>0$ such that the following condition holds 
    \begin{equation}\label{eq:GNprop}
     \text{for each }n\in\N , (\overline a_1,\ldots,\overline a_n)\in \lcA^n \text{ with } Q(\overline a_j)\leq 1 \text{ we have } P(\overline a_1 \cdots \overline a_n)\leq M^n.
    \end{equation}
  Now we set $q_0 \coloneq Q|_{\lcA_0}$ and observe that $q_0$ is a continuous seminorm as the inclusion $\lcA_0 \rightarrow \lcA$ is continuous and linear.
  A trivial computation now shows that $p_0, q_0$ and $M$ satisfy \eqref{eq:GNprop}. 
  We conclude that $\lcA_0$ has the (GN)-property.
  
  For the converse assume that $\lcA_0$ has the (GN)-property and fix a continuous seminorm $P$ on $\lcA$. 
  The topology on $\lcA$ is the product topology, i.e.\ it is generated by the canonical projections $\smfunc{\pi_N}{\lcA}{\lcA_N}$.
  Hence we may assume that $P$ is of the form
  \[
   P = \max_{0 \leq N \leq L} \left(	 p_N\circ \pi_N	\right)
  \]
  where each $\smfunc{p_N}{\lcA_N}{\left[0,\infty\right[}$ is a continuous seminorm on $\lcA_N$. 
  The number $L \in\N_0$ is finite and remains fixed for the rest of the proof.
  
  The key idea is here that $P$ depends only on a finite number of spaces in the grading.
  Now the multiplication increases the degree of elements except for factors of degree $0$.
  However, these contributions can be controlled by the (GN)-property in $\lcA_0$.  
  
  We will now construct a continuous seminorm $Q$ on $\lcA$ and a number $M>0$ such that \eqref{eq:GNprop} holds.
 \medskip
 
 \textbf{Construction of the seminorm $Q$.} 
  For a multi-index $w = (w_1 , w_2 , \ldots , w_r)\in\N^r$ (where $r \in \N$) we denote by $\abs{w}$ the sum of the entries of the multi-index $w$. 
  Define for $r\leq L$ and $w \in \N^r$ with $\abs{w} \leq L$ a continuous $r$-linear map
  \begin{displaymath}
   \func{f_w}{\lcA_{w_1}\times \lcA_{w_2} \times \cdots\times \lcA_{w_r}}{\lcA_{\abs{w}}}{(b_1,\ldots,b_r)}{b_1\cdots b_r,}
  \end{displaymath}
  Since $L<\infty$ is fixed and the $w_i$ are strictly positive for $1 \leq i \leq r$, there are only finitely many maps $f_w$ of this type. 
  This allows us to define for each $1 \leq k \leq L$ a seminorm $q_k$ on $\lcA_k$ with the following property: 
  For all $r\leq L$ and $w\in\N^r$ with $\abs{w} \leq L$ we obtain an estimate 
  \begin{equation}\label{eq: f_GN}
   p_{\abs{w}} \left(	f_w(b_1,\ldots,b_r)	\right) \leq q_{w_1}(b_1) \cdot q_{w_2}(b_2) \cdots q_{w_r}(b_r).
  \end{equation}
  Consider for each $N\leq L$ the continuous trilinear map
  \begin{displaymath}
   \func{g_N}{\lcA_0\times \lcA_N \times \lcA_0}{\lcA_N}{(c,d,e)}{c\cdot d\cdot e.}
  \end{displaymath}
  As there are only finitely many of these maps, we can define a seminorm $q_0$ on $\lcA_0$ and a seminorm $q^\sim_N$ on $\lcA_N$ for each $1 \leq N \leq L$ such that
  \begin{equation}\label{eq: g_GN}
   q_N\left(g_N(c,d,e) \right) \leq q_0(c) q^\sim_N(d) q_0(e) \hbox{ holds for }  1 \leq N \leq L.
  \end{equation}
  Enlarging the seminorm $q_0$, we may assume that $q_0\geq p_0$.

  Now we use the fact that the subalgebra $\lcA_0$ has the (GN)-property. 
  Hence, there is a continuous seminorm $q^\sim_0$ on $\lcA_0$ and a number $M_0\geq1$ such that for $n\in\N$ and elements $c_i \in \lcA_0,\ 1 \leq i \leq n$ with $q^\sim_0 (c_i) \leq 1$ we have
  \begin{equation}\label{eq: q_GN}
   q_0(c_1\cdots c_n)\leq M_0^n.
  \end{equation}
  Finally, we define the seminorm $Q$ via 
  \begin{displaymath}
   Q \coloneq \max_{0 \leq k \leq L} \left(q^\sim_k\circ \pi_k	\right).
  \end{displaymath}
  Clearly $Q$ is a continuous seminorm on $\lcA$. 
  Moreover, we set $M:=M_0^{2(L+1)}\cdot (L+1)$.
  \medskip
  
  \textbf{The seminorms $P$, $Q$ and the constant $M$ satisfy \eqref{eq:GNprop}.}
  Let $n\in\N$ and $(\overline a_1,\ldots,\overline a_n)\in A^n$ with $Q(\overline a_j)\leq 1$ be given. It remains to show that $P(\overline a_1 \cdots \overline a_n)\leq M^n$. Each element $\overline a_j$ can be written as a converging series
  \begin{displaymath}
   \overline a_j=\sum_{k=0}^\infty a_j^{(k)} \hbox{ with } a_j^{(k)}\in A_k.
  \end{displaymath}
  Plugging this representation into $P$, we obtain the estimate
  \begin{align*}
   P\left(	\overline a_1 \cdots \overline a_n	\right)
		  &	=	P\left(	\sum_{N=0}^\infty \sum_{\substack{\alpha\in\N_0^n \\ \abs{\alpha} = N}} a_1^{(\alpha_1)} \cdots a_n^{(\alpha_n)}	\right)
			=	\max_{0 \leq N \leq L} 	 p_N\left(	 \sum_{\substack{\alpha\in\N_0^n \\ \abs{\alpha} = N}} a_1^{(\alpha_1)} \cdots a_n^{(\alpha_n)}	\right)
		\\&	\leq	\max_{0 \leq N \leq L} 	 \sum_{\substack{\alpha\in\N_0^n \\ \abs{\alpha} = N}} p_N\left(	  a_1^{(\alpha_1)} \cdots a_n^{(\alpha_n)}	\right).
  \end{align*}
  For fixed $0 \leq N \leq L$ the number of summands in this sum is bounded from above by $(N+1)^n\leq (L+1)^n$ since for the entries of $\alpha$ there are at most $N+1$ choices. 
  We claim that each summand can be estimated as 
  \begin{equation}\label{eq: claim_GN}
   p_N\left(	  a_1^{(\alpha_1)} \cdots a_n^{(\alpha_n)}	\right)\leq \left(M_0^{2(L+1)}\right)^n.
  \end{equation}
  If this is true then one easily deduces that $P\left(	\overline a_1 \cdots \overline a_n	\right)\leq (L+1)^n\cdot (M_0^{2(L+1)})^n = M^n$ and the assertion follows.
  
  Hence we have to prove that \eqref{eq: claim_GN} holds. 
  To this end, fix $0 \leq  N \leq L$ and $\alpha\in\N_0^n$.
  \medskip 
  
  \noindent \textbf{Case $N=0$:} Then $\alpha=(0,\ldots,0)$ and we have
  \begin{align*}
		      p_0\left(	  a_1^{(0)} \cdots a_n^{(0)}	\right)   	\leq q_0\left(	  a_1^{(0)} \cdots a_n^{(0)}	\right)     	\stackrel{          \eqref{eq: q_GN} }{\leq}  M_0^n       	\leq \left(M_0^{2(L+1)}\right)^n.
  \end{align*}
  \textbf{Case $N\geq1$:} The product $a_1^{(\alpha_1)} \cdots a_n^{(\alpha_n)}$ may contain elements from the subalgebra $A_0$ and elements from the subspaces $A_k$ with $k\geq1$.
  Combining each element contained in $A_0$ with elements to the left or the right, we rewrite the product as
  \[
   a_1^{(\alpha_1)} \cdots a_n^{(\alpha_n) } =  b_1\cdots b_r
  \]
  for some $r\leq \min \{ n , L\}$. 
  Deleting all zeroes from $\alpha$, we obtain a multi-index $w \in \N^r$.
  Now by construction $b_k\in A_{w_k}$ is a product $b_k=c_k\cdot d_k \cdot e_k$, where each $d_k\in A_{w_k}$ is one of the $a_j$ and $c_k$ and $e_k$ are finite products of $\lcA_0$-factors in the original product.
  
  Since each $c_k$ is a product of at most $n$ elements of $\lcA_0$, all of which have $q_0^\sim$-norm at most $1$, we may apply \eqref{eq: q_GN} to obtain the estimate:
  \[
   q_0(c_k)\leq M_0^n.
  \]
  For the same reason, we have the corresponding estimate $q_0(e_k)\leq M_0^n$.

  Combining these results, we derive
  \begin{align*}
   p_N \left(a_1^{(\alpha_1)} \cdots a_n^{(\alpha_n) }\right)   &= p_N(  b_1\cdots b_r)  \stackrel{\eqref{eq: f_GN}}{\leq}   \prod_{k=1}^r q_{w_k}(b_k) =	\prod_{k=1}^r q_{w_k}(c_k d_k e_k)
	\\& \stackrel{\eqref{eq: g_GN}}{\leq}  	\prod_{k=1}^r \underbrace{q_0(c_k)}_{\leq M_0^n} \cdot \underbrace{q^\sim_{w_k}(d_k)}_{\leq1} \cdot \underbrace{q_0(e_k)}_{\leq M_0^n} \leq  \prod_{k=1}^r (M_0^2)^n =    (M_0^{2r})^n
	\\& \leq    \left(M_0^{2(L+1)}\right)^n\qedhere
  \end{align*}
 \end{proof}

\section{Characters on graded connected Hopf algebras}														\label{section: characters_on_graded_Hopf_algebras_are_Lie_groups}
 In this section we construct Lie group structures on character groups of (graded and connected) Hopf algebras.
 
 \begin{setup}	  
 Throughout this section, let $\Hopf=(\Hopf,m_\Hopf,u_\Hopf,\Delta_\Hopf,\epsilon_\Hopf,S_\Hopf)$ be a fixed Hopf algebra and let $\lcB$ be a fixed \emph{commutative} locally convex algebra.

 Using only the coalgebra structure of $\Hopf$, we obtain the locally convex algebra 
 \[
  \lcA\coloneq(\Hom_\K(\Hopf,\lcB),\star) \quad \text{(see \ref{setup: maps on coalgebra})}.
 \]
\end{setup}

 Note that our framework generalises the special case $\lcB=\K$ which is also an interesting case. 
 For example, the Hopf algebra of rooted trees (see Example \ref{ex: RT}) is a connected, graded Hopf algebra and its group of $\K$-valued characters turns out to be the Butcher group from numerical analysis (cf.\ Example \ref{ex: Butcher}).\medskip

 We will now consider groups of characters of Hopf algebras:
 \begin{defn}							\label{def: character}
  A linear map $\smfunc{\phi}{\Hopf}{\lcB}$ is called ($\lcB$-valued) \emph{character} if it is a homomorphism of unital algebras, i.e.
  \begin{equation}\label{eq char:char}
   \phi(a b) = \phi(a)\phi(b)\text{ for all } a,b \in \Hopf \text{ and } \phi(1_\Hopf)=1_\lcB.
  \end{equation}
  Another way of saying this is that $\phi$ is a character, if
  \begin{equation}\label{eq: char:abs}
   \phi\circ m_\Hopf = m_\lcB \circ (\phi\otimes \phi) \hbox{ and } \phi(1_\Hopf)=1_\lcB.
  \end{equation}
  The set of characters is denoted by $\Char{\Hopf}{\lcB}$. 
 \end{defn}

 \begin{lem}\label{lem: char:mult}
  The set of characters $\Char{\Hopf}{\lcB}$ is a closed subgroup of $(\lcA^\times,\star)$. With the induced topology, $\Char{\Hopf}{\lcB}$ is a topological group.
  Inversion in this group is given by the map $\phi \mapsto \phi \circ S_\Hopf$ and the unit element is $1_\lcA\coloneq u_\lcB\circ\epsilon_\Hopf \colon \Hopf \rightarrow \lcB , x \mapsto \epsilon_\Hopf (x) 1_\lcB$.
 \end{lem}
 \begin{proof}
  The fact that the characters of a Hopf algebra form a group with respect to the convolution product is well-known, see for example \cite[Proposition II.4.1 3)]{Manchon}. Note that in loc.cit. this is only stated for a connected graded Hopf algebra although the proof does not use the grading at all.

  The closedness of $\Char{\Hopf}{\lcB}$ follows directly from Definition \ref{def: character} and the fact that we use the topology of pointwise convergence on $\lcA$. Continuity of the convolution product was shown in \ref{setup: maps on coalgebra}. Inversion is continuous as the precomposition with the antipode is obviously continuous with respect to pointwise convergence. 
 \end{proof}

 Our goal in this section is to turn the group of characters into a Lie group. Hence, we need a modelling space for this group. This leads to the following definition:

 \begin{defn}
  A linear map $\phi\in\Hom_\K(\Hopf,\lcB)$ is called an \emph{infinitesimal character} if
  \begin{equation}\label{eq: InfChar:char}
   \phi\circ m_\Hopf = m_\lcB\circ (\phi\otimes \epsilon_\Hopf + \epsilon_\Hopf \otimes \phi),
  \end{equation}
  which means for $a,b \in \Hopf$ that $\phi(a b)=\phi(a) \epsilon_\Hopf(b) + \epsilon_\Hopf(a) \phi(b)$.
  
  We denote by  $\InfChar{\Hopf}{\lcB}$ the set of all infinitesimal characters. 
 \end{defn}

  \begin{lem}								\label{lem: inf:subalg}
   The infinitesimal characters $\InfChar{\Hopf}{\lcB}$ form a closed Lie subalgebra of $(\lcA,[\cdot,\cdot])$, where $\LB{}$ is the commutator bracket of $(A,\star)$. 
  \end{lem}
  \begin{proof}
   As for $\Char{\Hopf}{\lcB}$, the closedness follows directly from the definition. The fact that the infinitesimal characters form a Lie subalgebra is well-known, see for example \cite[Proposition II.4.2]{Manchon}. 
  \end{proof}

  \begin{setup}								\label{setup: infcharIdealA}
   From now on we assume for the rest of this section that the Hopf algebra $\Hopf$ is \emph{graded and connected} (see \ref{setup: abstract grading}). 
   Thus by Lemma \ref{lem: grading and CIA} the locally convex algebra $\lcA=\Hom_\K(\Hopf,\lcB)$ is densely graded.
  
   Every infinitesimal character $\phi\in\InfChar{\Hopf}{\lcB}$ maps $1_\Hopf$ to $0_\K \cdot 1_\lcB$ since $\phi(1_\Hopf\cdot 1_\Hopf) = \phi(1_\Hopf) \epsilon_\Hopf(1_\Hopf) + \epsilon_\Hopf(1_\Hopf) \phi(1_\Hopf) = 2\phi (1_\Hopf)$.
   Now $\Hopf=\bigoplus_{n=0}\Hopf_n$ is assumed to be connected and we have $\Hopf_0=\K 1_\Hopf$, whence $\phi|_{\Hopf_0}=0$. 
   Translating this to the densely graded algebra $\lcA$, we observe $\InfChar{\Hopf}{\lcB}\subseteq \IdealA$.

   Similarly, a character maps $1_\Hopf$ to $1_\lcB$ by definition. Hence, $\Char{\Hopf}{\lcB}\subseteq 1_\lcA+\IdealA$. 
  \end{setup}

 \begin{thm}
 \label{thm: Char:Lie}
  Let $\Hopf$ be an abstract graded connected Hopf algebra $\Hopf$.
  For any commutative locally convex algebra $\lcB$, the group $\Char{\Hopf}{\lcB}$ of $\lcB$-valued characters of $\Hopf$ is a ($\K-$analytic) Lie group.
  
  Furthermore, we observe the following properties 
  \begin{itemize}
    \item[\textup{(i)}] The Lie algebra $\Lf (\Char{H}{\lcB})$ of $\Char{H}{\lcB}$ is the Lie algebra $\InfChar{\Hopf}{\lcB}$ of infinitesimal characters with the commutator bracket $\LB[\phi,\psi] = \phi \star \psi - \psi \star \phi$.
    \item[\textup{(ii)}] $\Char{\Hopf}{\lcB}$ is a BCH--Lie group which is exponential, i.e.~the exponential map is a global $\K$-analytic diffeomorphism and is given by the exponential series.
    \item[\textup{(iii)}] The model space of $\Char{\Hopf}{\lcB}$ is a \Frechet space whenever $\Hopf$ is of countable dimension (e.g.~a Hopf algebra of finite type) and $\lcB$ is a \Frechet space (e.g.~a Banach algebra or a finite  
     dimensional algebra).
     
     In the special case that $\lcB$ is a weakly complete algebra, the modelling space $\InfChar{\Hopf}{\lcB}$ is weakly complete as well.
    \end{itemize}
 \end{thm}
 
 \begin{rem}
  Note that character groups of (graded and connected) Hopf algebras with values in locally convex algebras arise naturally in application in physics.
  Namely, in renormalisation of quantum field theories one considers characters with values in algebras of polynomials or the algebra of germs of meromorphic functions (see e.g.\ \cite{Manchon}).
 \end{rem}
 
 \newcommand{\EXPISO}{\mathrm{Exp}}
 \begin{proof}[Proof of Theorem \ref{thm: Char:Lie}]
  The locally convex algebra $\lcA=\Hom_\K(\Hopf,\lcB)$ is densely graded. Hence, by Lemma \ref{lem: exp_and_log}, the exponential series converges on the closed vector subspace $\IdealA$ and defines a $C^\omega_\K$-diffeomorphism:
  \[
   \func{\exp_\lcA}{\IdealA}{1_\lcA + \IdealA}{\phi}{\exp[\phi].}
  \]
  This implies that the closed vector subspace $\InfChar{\Hopf}{\lcB}$ is mapped to a closed analytic submanifold of $1_\lcA+\IdealA\subseteq\lcA$. 
  By Lemma \ref{lem: exp:bij}, we obtain a commutative diagram
  \begin{equation}\label{eq: comm_diag_exp}
      \begin{aligned}
      \xymatrix{
		  \IdealA \ar[rrr]^{\exp_\lcA}											& & &	1_\lcA + \IdealA 		\\
		  \InfChar{\Hopf}{\lcB} \ar[u]^{\subseteq} \ar[rrr]_{\EXPISO\coloneq\exp_\lcA|_{\InfChar{\Hopf}{\lcB}}^{\Char{\Hopf}{\lcB}}}		& & &	\Char{\Hopf}{\lcB} \ar[u]_{\subseteq}			}
      \end{aligned}
  \end{equation}
  which shows that the group $\Char{\Hopf}{\lcB}$ is a closed analytic submanifold of $1_{\lcA}+\IdealA\subseteq\lcA$.

  The group multiplication is a restriction of the continuous bilinear map $\smfunc{\star}{\lcA\times\lcA}{\lcA}$ to the analytic submanifold $\Char{\Hopf}{\lcB}\times \Char{\Hopf}{\lcB}$ and since the range space $\Char{\Hopf}{\lcB}\subseteq \lcA$ is closed this ensures that the restriction is analytic as well.
  Inversion in $\Char{\Hopf}{\lcB}$ is composition with the antipode map $\nnfunc{\Char{\Hopf}{\lcB}}{\Char{\Hopf}{\lcB}}{\phi}{\phi\circ S}$ (see Lemma \ref{lem: char:mult}). This is a restriction of the continuous linear (and hence analytic) map $\nnfunc{\lcA}{\lcA}{\phi}{\phi\circ S}$ to closed submanifolds in the domain and range and hence inversion is analytic as well. 
  This shows that $\Char{\Hopf}{\lcB}$ is an analytic Lie group.

  \begin{itemize}
   \item [(i)] 
    The map $\func{\EXPISO}{\InfChar{\Hopf}{\lcB}}{\Char{\Hopf}{\lcB}}{\phi}{\exp_\lcA(\phi)}$ is an analytic diffeomorphism.
    Hence, we take the tangent map at point $0$ in the diagram \eqref{eq: comm_diag_exp} to obtain: 

  \[
   \xymatrix{
		  \IdealA \ar[rrr]^{T_0\exp_\lcA}				& & &	\IdealA 		\\
		  \InfChar{\Hopf}{\lcB} \ar[u]^\subseteq \ar[rrr]_{T_0\EXPISO}	& & &	 T_{1_\lcA}(\Char{\Hopf}{\lcB})  \ar[u]_{\subseteq}			}
  \]
  By Lemma \ref{lem: exp:com} (b) we know that $T_0\exp_\lcA=\id_{\lcA}$ which implies $T_0\EXPISO=\id_{\InfChar{\Hopf}{\lcB}}$.
  Hence as locally convex spaces $\Lf(\Char{\Hopf}{\lcB})=T_{1_{\lcA}}(\Char{\Hopf}{\lcB}) = \InfChar{\Hopf}{\lcB}$.
  It remains to show that the Lie bracket on $T_{1_{\lcA}}(\Char{\Hopf}{\lcB})$ induced by the Lie group structure is the commutator bracket on $\InfChar{\Hopf}{\lcB}$. This is what we will show in the following:

  For a fixed $\phi\in \Char{\Hopf}{\lcB}$, the inner automorphism $\func{c_\phi}{\Char{\Hopf}{\lcB}}{\Char{\Hopf}{\lcB}}{\psi}{\phi\star\psi\star(\phi\circ S)}$ is a restriction of continuous linear map on the ambient space $\lcA$.
  Taking the derivative of $c_\phi$ we obtain the usual formula for the adjoint action
  \[
   \func{\Lf(c_\phi)=\Ad(\phi)}{\InfChar{\Hopf}{\lcB}}{\InfChar{\Hopf}{\lcB}}{\psi}{\phi\star\psi\star(\phi\circ S).}
  \]
  For fixed $\psi\in\InfChar{\Hopf}{\lcB}$, this formula can be considered as a restriction of a continuous polynomial in $\phi$. 
  We then define $\ad \coloneq T_{1_\lcA} \Ad (\cdot).\psi$ (cf.\ \cite[Example II.3.9]{MR2261066}) which yields the Lie bracket of $\phi$ and $\psi$ in $\InfChar{\Hopf}{\lcB}$:
  \begin{equation}\label{eq: LB}
   [\phi , \psi ] = \ad(\phi).\psi = \phi\star\psi + \psi\star (\phi\circ S).
  \end{equation}
  Recall that the antipode is an anti-coalgebra morphism by \cite[Proposition 1.3.1]{MR1381692}. 
  Thus the map $\lcA \rightarrow \lcA , f \mapsto f \circ S$ is an anti-algebra morphism which is continuous with respect to the topology of pointwise convergence. 
  We conclude for an infinitesimal character $\phi$ that  $\exp_\lcA (\phi \circ S) = \exp_\lcA (\phi) \circ S = (\exp_\lcA (\phi))^{-1}$, i.e.\ $\exp_\lcA (\phi \circ S)$ is the inverse of $\exp_\lcA (\phi)$ with respect to $\star$. 
  As $\exp_\lcA$ restricts to a bijection from $\InfChar{\Hopf}{\lcB}$ to $\Char{\Hopf}{\lcB}$ we derive from Lemma \ref{lem: exp:com} (a) that $\phi \circ S = - \phi$.
  This implies together with \eqref{eq: LB} that $\LB[\phi,\psi] = \phi \star \psi - \psi \star \phi$.
  
  \item[(ii)]
  We already know that the exponential series defines an analytic diffeomorphism $\smfunc{\EXPISO}{\InfChar{\Hopf}{\lcB}}{\Char{\Hopf}{\lcB}}$ between Lie algebra and Lie group. It only remains to show that $\EXPISO$ is the exponential function of the Lie group. To this end let $\phi\in\InfChar{\Hopf}{\lcB}$ be given. 
  The analytic curve
  \[
   \func{\gamma_{\phi}}{(\R,+)}{\Char{\Hopf}{\lcB}}{t}{\EXPISO(t \phi)}
  \]
  is a group homomorphism by Lemma \ref{lem: exp:com} (a) and we have $\gamma_\phi'(0)=\phi$ by Lemma \ref{lem: exp:com} (b).
  Hence by definition of a Lie group exponential function (see \cite[Definition II.5.1]{MR2261066}) we have $\exp(\phi)=\gamma_\phi(1)=\EXPISO(\phi)$.

  \item[(iii)]
  Assume that the underlying vector space of the algebra $\Hopf$ is of countable dimension.
  The Lie algebra $\InfChar{\Hopf}{\lcB}$ is a closed vector subspace of $\lcA$, the latter is \Frechet by Lemma \ref{lem: completeness of A} (b). 
  Hence, $\InfChar{\Hopf}{\lcB}$ is \Frechet as well.

  If $\lcB$ is weakly complete then $\InfChar{\Hopf}{\K}$ is a closed vector subspace of a weakly complete space by Lemma \ref{lem: completeness of A} (a). 
  Since closed subspaces of weakly complete spaces are again weakly complete (e.g.~\cite[Theorem A2.11]{MR2337107}), the assertion follows.\qedhere
 \end{itemize}
 \end{proof}

 We now turn to the question whether the Lie group constructed in Theorem \ref{thm: Char:Lie} is a \emph{regular} Lie group. 
 To derive the regularity condition, we need to restrict our choice of target algebras to \emph{Mackey complete} algebras. 
 Let us note first the following useful result on character groups.
 
 \begin{lem}\label{lem: Char:inCIA}
  Let $\Hopf$ be a graded and connected Hopf algebra and $\lcB$ be a commutative locally convex algebra. 
  Then $\Char{\Hopf}{\lcB}$ and $\InfChar{\Hopf}{\lcB}$ are contained in a closed subalgebra of $\lcA = (\Hom_\K (\Hopf , \lcB) , \star)$ which is a densely graded continuous inverse algebra.
 \end{lem}
 
 \begin{proof}
  Consider the closed subalgebra $\Cut{\lcA} \coloneq (\K 1_{\lcA_0}) \times \prod_{n\in \N} \lcA_n$.
  By Lemma \ref{lemma: red2CIA} this algebra is densely graded (with respect to the grading induced by $\lcA$) and a continuous inverse algebra.
  Note that $\IdealA = \ker \pi_0$ and $\pi_0 (1_\lcA + \IdealA) = 1_{\lcA}$ imply that $\IdealA$ and $1_{\lcA} + \IdealA$ are contained in $\Cut{\lcA}$.
  Thus \ref{setup: infcharIdealA} shows that $\Cut{\lcA}$ contains $\Char{\Hopf}{\lcB}$ and $\InfChar{\Hopf}{\lcB}$.
 \end{proof}

  \begin{thm}													\label{thm: Char_regular}
   Let $\Hopf$ be a graded and connected Hopf algebra and $\lcB$ be a Mackey complete locally convex algebra. 
   \begin{itemize}
    \item [\textup{(a)}] The Lie group $\Char{\Hopf}{\lcB}$ is $C^1$-regular.
    \item [\textup{(b)}] If in addition, $\lcB$ is sequentially complete, then $\Char{\Hopf}{\lcB}$ is even $C^0$-regular.
   \end{itemize}
   In both cases, the associated evolution map is even a $\K$-analytic map. 
   \end{thm}

  \begin{proof}
   We have seen in Lemma \ref{lem: Char:inCIA} that $\Char{\Hopf}{\lcB}$ and $\InfChar{\Hopf}{\lcB}$ are contained in a closed subalgebra $\Lambda$ of $\lcA = (\Hom_\K (\Hopf,\lcB), \star)$.
   By construction $\Lambda$ is a densely graded CIA.
   Note that $\Lambda$ inherits all completeness properties from $\lcB$ since it is closed in $\lcA$ and this space inherits its completeness properties from $\lcB$ by Lemma \ref{lem: completeness of A}. 
   Combine Lemma \ref{lem: char:mult} and Lemma \ref{lem: unit_groups_of_graded_algebras} (a) to see that $\Char{\Hopf}{\lcB}$ is a closed subgroup of the group of units of the CIA $\Lambda$.
   Theorem \ref{thm: Char:Lie} and \ref{thm: glockner_CIA_BCH} show that $\Lambda^\times$ and $\Char{\Hopf}{\lcB}$ are BCH--Lie groups and thus \cite[Theorem IV.3.3]{MR2261066} entails that $\Char{\Hopf}{\lcB}$ is a closed Lie subgroup of $\Lambda^\times$.
   
   Since $\Lambda_0 \cong \K$ (cf.\ Lemma \ref{lemma: red2CIA}) is a commutative CIA, we deduce from \ref{setup: GN_commutative_locally_m_convex} that $\Lambda_0$ has the (GN)-property.
   Hence Lemma \ref{lem: densely_GN} shows that also $\Lambda$ has the (GN)-property.  
   We deduce from \ref{setup: GN_regular} that the unit group $\Lambda^\times$ is a $C^1$-regular Lie group which is even $C^0$-regular if $\lcB$ is sequentially complete. 
   For the rest of this proof, we fix $k\in\smset{0,1}$ and assume that $\lcB$ is sequentially complete in the case $k=0$. 
   Let us prove now that the $C^k$-regularity of $\Lambda^\times$ implies that $\Char{\Hopf}{\lcB}$ is $C^k$-regular.

   Our first goal is to show that $\Char{\Hopf}{\lcB}$ is $C^k$-semiregular, i.e.\ that every $C^k$-curve into the Lie algebra $\InfChar{\Hopf}{\lcB}$ admits a $C^{k+1}$-evolution in the character group (cf.\ \cite{1208.0715v3}). 
   To this end, let us recall the explicit regularity condition for the unit group.\medskip
   
   \textbf{Step 1: The initial value problem for $C^k$-regularity in the group $\Lambda^\times$.}
   The Lie group $\Lambda^\times$ is open in the CIA $\Lambda$. 
   Take the canonical identification $T\Lambda^\times \cong \Lambda^\times \times \Lambda$. 
   Recall that the group operation of $\Lambda^\times$ is the restriction of the bilinear convolution $\star \colon \Lambda \times \Lambda \rightarrow \Lambda$.
   Consider for $\theta \in \Lambda^\times$ the left translation $\lambda_\theta (h) \coloneq \theta \star h, h \in \Lambda^\times$. 
   Then the identification of the tangent spaces yields $T_{1_\Lambda} \lambda_\theta (X) = \theta \star X$ for all $X \in T_{1_\Lambda} \Lambda^\times = \Lambda$. 
   Summing up, the initial value problem associated to $C^k$-regularity of $\Lambda^\times$ becomes 
     \begin{equation}\label{eq: diff}
     \begin{cases}
      \eta' (t) &= \eta (t) . \gamma (t) = T_{1_\Lambda} \lambda_{\eta (t)} (\gamma (t)) = \eta(t) \star \gamma (t) \quad t \in [0,1],\\
      \eta (0) &= 1_\Lambda
     \end{cases}
   \end{equation}
   where $\gamma \in C^k ([0,1], \Lambda)$. \medskip
   
   Fix $\gamma \in C^k ([0,1] , \InfChar{\Hopf}{\lcB})$. 
   Now $\InfChar{\Hopf}{\lcB} \subseteq \Lambda$ holds and $\Lambda^\times$ is $C^k$-regular. 
   Thus $\gamma$ admits a $C^{k+1}$-evolution $\eta$ in $\Lambda^\times$, i.e.\ $\eta \colon [0,1] \rightarrow \Lambda^\times$ is of class $C^{k+1}$ and solves \eqref{eq: diff} with respect to $\gamma$.
   We will now show that $\eta$ takes its values in $\Char{\Hopf}{\lcB}$.\medskip
   
   \textbf{Step 2: An auxiliary map to see that $\Char{\Hopf}{\lcB}$ is $C^k$-semiregular.}
   Consider 
   \begin{displaymath}
    F \colon [0,1] \times \Hopf \times \Hopf \rightarrow \lcB , (t,x,y) \mapsto \eta (t) (xy) - \eta (t) (x) \eta (t) (y).
   \end{displaymath}
   If $F$ vanishes identically, the evolution $\eta$ is multiplicative for each fixed $t$. 
   Note that as $\eta$ is a $C^{k+1}$-curve and $\Lambda$ carries the topology of pointwise convergence, for each $(x,y) \in \Hopf \times \Hopf$ the map $F_{x,y} \coloneq F (\cdot , x,y) \colon [0,1] \rightarrow \lcB$ is a $C^{k+1}$-map.
   Furthermore, $F(0, \cdot, \cdot) \equiv 0$ as $\eta (0) = 1_\Lambda \in \Char{\Hopf}{\lcB}$. 
   Thus for each pair $(x,y) \in \Hopf \times \Hopf$ the fundamental theorem of calculus yields 
    \begin{equation}\label{eq: funthm}
     F(t,x,y) = F_{x,y} (t) = \int_0^t \frac{\partial}{\partial t} F_{x,y} (t) \dd t.
    \end{equation}
  To evaluate this expression we compute the derivative of $F_{x,y}$ as
    \begin{equation}
    \begin{aligned}
     \frac{\partial}{\partial t} F_{x,y} (t) 
		&\stackrel{\hphantom{\eqref{eq: diff}}}{=} \frac{\partial}{\partial t} \eta (t) (xy) - \left(\frac{\partial}{\partial t} \eta (t) (x)\right) \eta (t) (y) - \left(\frac{\partial}{\partial t} \eta (t) (y)\right) \eta (t) (x) \\ 
		&\stackrel{\eqref{eq: diff}}{=} [\eta (t) \star \gamma (t)] (xy) - [\eta (t) \star \gamma (t)] (x) \eta (t) (y) - [\eta (t) \star \gamma (t)] (y) \eta (t) (x).
     \end{aligned} \label{eq: diff:aux} 
    \end{equation}
  In the following formulae, abbreviate $\eta_t \coloneq \eta (t)$ and $\gamma_t \coloneq \gamma (t)$ to shorten the notation.
  We use Sweedler's sigma notation to write $\Delta_\Hopf (x) = \sum_{(x)} x_{1}\otimes x_{2}$ and $\Delta_\Hopf (y) = \sum_{(y)} y_{1} \otimes y_{2}$.
  As $\Delta_\Hopf$ is an algebra homomorphism, the convolution in \eqref{eq: diff:aux} can then be rewritten as 
    \begin{equation}\label{eq: diff:aux2} 
    \frac{\partial}{\partial t} F_{x,y} (t) = \sum_{(x)(y)} \eta_t (x_1y_1) \gamma_t (x_2y_2) - \sum_{(x)} \eta_t (x_1) \gamma_t (x_2) \eta_t (y) - \sum_{(y)} \eta_t (y_1) \gamma_t (y_2) \eta_t (x) .
    \end{equation}
  Recall that the curve $\gamma$ takes its values in the infinitesimal characters, whence we have the identity $\gamma_t (ab) = \epsilon (a) \gamma_t (b) + \epsilon (b) \gamma_t (a)$.
  Plugging this into the first summand in \eqref{eq: diff:aux2} and using that $\eta_t$ is linear for all $t$ we obtain the identity 
    \begin{equation} \begin{aligned}
                       \sum_{(x)} \sum_{(y)} \eta_t (x_1y_1) \gamma_t (x_2y_2)& = \sum_{(x)} \sum_{(y)} (\eta_t (\epsilon (x_2)x_1y_1) \gamma_t (y_2) +  \eta_t (x_1 (\epsilon (y_2)y_1)) \gamma_t (x_2)) \\ 
									  & = \sum_{(y)} \eta_t (xy_1)\gamma_t (y_2) + \sum_{(x)} \eta_t (x_1y) \gamma_t (x_2).
                     \end{aligned}\label{eq: diff:aux3} 
     \end{equation}
 As $\lcB$ is commutative inserting \eqref{eq: diff:aux3} into \eqref{eq: diff:aux2} yields 
    \begin{equation}\label{eq: diff:auxfin}
     \frac{\partial}{\partial t} F_{x,y} (t) = \sum_{(y)} (\eta_t (xy_1) - \eta_t (x)\eta_t(y_1)) \gamma_t (y_2) + \sum_{(x)} (\eta_t (x_1y) - \eta_t (x_1)\eta_t(y)) \gamma_t (x_2).
    \end{equation}
 Since $\eta_t$ is linear for each fixed $t$ it suffices to check that $\eta_t$ is multiplicative for all pairs of elements in a set spanning the vector space $\Hopf$. 
 As $\Hopf$ is graded, the homogeneous elements span the vector space $\Hopf$.
 We will now use the auxiliary mapping $F$ and its partial derivative to prove that the evolution $\eta_t$ is multiplicative on all homogeneous elements in $\Hopf$ (whence on all elements in $\Hopf$). 
 \medskip
   
 \textbf{Step 3: The evolution $\eta (t)$ is multiplicative on $\Hopf_0$ and maps $1_\Hopf$ to $1_\lcB$.} \\
  The Hopf algebra $\Hopf$ is graded and connected, i.e.\ $\Hopf = \bigoplus_{n \in \N_0} \Hopf_n$ and $\Hopf_0 = \K 1_\Hopf$. 
  By construction this entails $\Delta_\Hopf (\Hopf_0) \subseteq \Hopf_0 \otimes \Hopf_0$ and the infinitesimal character $\gamma$ vanishes on $\Hopf_0$. 
  Thus for $x,y \in \Hopf_0$ we have $\frac{\partial}{\partial t} F_{x,y} (t) = 0$ for all $t \in [0,1]$ by \eqref{eq: diff:auxfin}. 
  We conclude from \eqref{eq: funthm} for $x,y \in \Hopf_0$ the formula  
 \begin{displaymath}
  \eta(t) (xy) - \eta (t) (x)\eta (t) (y) = F(t,x,y) = 0 \quad \forall t \in [0,1] .
 \end{displaymath}
 Hence $\eta (t) (xy) = \eta (t) (x) \eta (t) (y)$ for all elements in degree $0$. 
 Specialising to $x= 1_\Hopf = y$ we see that $\eta (t) (1_\Hopf)$ is an idempotent in the CIA $\lcB$. 
 Furthermore, since $\eta (t) \in \Lambda^\times$, we have $1_{\lcB} = 1_\Lambda (1_\Hopf) = \eta (t) \star (\eta (t))^{-1} (1_\Hopf)$, whence $\eta (t) (1_\Hopf) \in \lcB^\times$ for all $t \in [0,1]$. 
 As $\lcB^\times$ is a group it contains only one idempotent, i.e.\ $\eta (t) (1_\Hopf) = 1_{\lcB}$.
 \medskip
 
 \textbf{Step 4: The evolution $\eta (t)$ is multiplicative for all homogeneous elements.}
 As $\Hopf$ is connected, $\eta_t$ is linear and $\eta_t (1_\Hopf) = 1_\lcB$ holds, we see that \eqref{eq: diff:auxfin} vanishes if either $x$ or $y$ are contained in $\Hopf_0$. 
 We conclude from \eqref{eq: funthm} that 
  \begin{displaymath}
   \eta_t (xy) = \eta_t (x) \eta_t (y)\ \forall t \in [0,1] \text{ if } x \text{ or } y \text{ are contained in degree } 0
  \end{displaymath}
 Denote for a homogeneous element $x \in \Hopf$ by $\deg x$ its degree with respect to the grading.
 To prove that $\eta_t$ is multiplicative for elements of higher degree and $t \in [0,1]$ we proceed by induction on the sum of the degrees of $x$ and $y$.
 Having established multiplicativity of $\eta_t$ if at least one element is in $\Hopf_0$, we have already dealt with the cases $\deg x + \deg y \in \{0,1\}$. 
 \medskip
 
 \emph{Induction step for $\deg x + \deg y \geq 2$.} We assume that for homogeneous elements $a,b$ with $\deg a + \deg b \leq \deg x + \deg y -1$ the formula $\eta_t (ab) = \eta_t (a)\eta_t (b)$ holds.
  
 Since $\Hopf$ is connected, for each $z \in \Hopf_n$ with $n \geq 1$ the coproduct can be written as 
  \begin{displaymath}
   \Delta_\Hopf (z) = z \otimes 1_\Hopf + 1_\Hopf \otimes z + \tilde{\Delta} (z)
  \end{displaymath}
 where $\tilde{\Delta} (z) = \sum_{(\widetilde{z})} \tilde{z}_1 \otimes \tilde{z}_2 \in \epsilon_\Hopf^{-1} (0) \otimes \epsilon_\Hopf^{-1} (0)$ is the reduced coproduct. 
 Note that by construction the elements $\tilde{z}_1, \tilde{z}_2$ are homogeneous of degree strictly larger than $0$.
 Let us plug this formula for the coproduct into \eqref{eq: diff:auxfin}.
 We compute for the first sum in \eqref{eq: diff:auxfin}: 
  \begin{equation}\begin{aligned}
  &\sum_{(y)} (\eta_t (xy_1) - \eta_t (x)\eta_t(y_1)) \gamma_t (y_2) \\
  =& (\eta_t (xy)-\eta_t (x) \eta_t (y))\underbrace{\gamma_t(1_\Hopf)}_{=0} + \underbrace{(\eta_t (x1_\Hopf) -\eta_t (x) \eta_t (1_\Hopf))}_{=0}\gamma_t (y)\\
  & \hphantom{(\eta_t (xy)-\eta_t (x) \eta_t (y))\underbrace{\gamma_t(1_\Hopf)}_{=0}} + \sum_{\widetilde{(y)}} (\eta_t (x\tilde{y}_1) - \eta_t (x)\eta_t(\tilde{y}_1)) \gamma_t (\tilde{y}_2)\\
  =& \sum_{\widetilde{(y)}} \underbrace{(\eta_t (x\tilde{y}_1) - \eta_t (x)\eta_t(\tilde{y}_1))}_{C_{x,\tilde{y}_1} \coloneq} \gamma_t (\tilde{y}_2) \end{aligned} \label{eq: hom:weg}
  \end{equation}
 By construction we have $\deg \tilde{y}_1 , \deg \tilde{y}_2 \in [1,\deg y -1]$. 
 Then $\deg \tilde{y}_1 < \deg y$ implies $\deg x + \deg \tilde{y}_1 < \deg x +\deg y$ and thus $C_{x,\tilde{y}_1}$ vanishes by the induction assumption.
 
 As the two sums in \eqref{eq: diff:auxfin} are symmetric, interchanging the roles of $x$ and $y$ together with an analogous argument as above shows that also the second sum vanishes. 
 Hence, arguing as in Step 3, we see that $\eta_t (xy) = \eta_t (x)\eta_t (y)$ holds for all homogeneous elements $x,y \in \Hopf$. 
 
 In conclusion, the evolution $\eta \colon [0,1] \rightarrow \Lambda^\times$ of $\gamma$ takes its values in the closed subgroup $\Char{\Hopf}{\lcB}$ and thus $\Char{\Hopf}{\lcB}$ is $C^k$-semiregular.  
 \medskip 
 
 \textbf{Step 5: $\Char{\Hopf}{\lcB}$ is $C^k$-regular.} 
 Let $\iota \colon \InfChar{\Hopf}{\lcB} \rightarrow \Lambda$ be the canonical inclusion mapping. 
 Consider $\iota_* \colon C^k([0,1],  \InfChar{\Hopf}{\lcB}) \rightarrow ([0,1],\Lambda), c \mapsto \iota \circ c$.
 As $\iota$ is continuous and linear the map $\iota_*$ is continuous and linear by \cite[Lemma 1.2]{MR2997582}, whence smooth and even $\K$-analytic. 
 Let $\evol_\Lambda \colon C^k([0,1], \Lambda) \rightarrow \Lambda^\times$ be the (smooth) evolution map of the $C^k$-regular Lie group $\Lambda^\times$. 
 Then the map 
  \begin{displaymath}
   \evol \colon C^k([0,1],  \InfChar{\Hopf}{\lcB}) \rightarrow \Lambda^\times , \evol_{\Lambda^\times} \circ \iota_*
  \end{displaymath}
 is $\K$-analytic by Theorem \ref{thm: unit_group_regular} and maps a $C^k$-curve in the Lie algebra of $\Char{\Hopf}{\lcB}$ to its time $1$ evolution. 
 As the closed subgroup $\Char{\Hopf}{\lcB}$ is $C^k$-semiregular by Step 4, $\evol$ factors through a $\K$-analytic map 
  $ \evol_{\Char{\Hopf}{\lcB}} \colon C^k([0,1],  \InfChar{\Hopf}{\lcB}) \rightarrow \Char{\Hopf}{\lcB}$.
 Summing up, $\Char{\Hopf}{\lcB}$ is $C^k$-regular and the evolution map is $\K$-analytic.
  \end{proof}

\section{Subgroups associated to Hopf ideals}

So far, we were only able to turn the character group of a graded connected Hopf algebra $\Hopf$ into a Lie group. In this section, we will show that the character group of a quotient $\Hopf/\HIdeal$ can be regarded as a closed Lie subgroup of the character group of $\Hopf$ and thus carries a Lie group structure as well. It should be noted that this does not imply that the character group of \emph{every} Hopf algebra can be endowed with a Lie group structure (see Example \ref{ex: group algebra}).

\begin{defn}[Hopf ideal] \label{defn: Hideal}
Let $\Hopf$ be a Hopf algebra. 
We say $\HIdeal\subseteq \Hopf$ is a \emph{Hopf ideal} if the subset $\HIdeal$ is
\begin{enumerate}
\item a two-sided (algebra) ideal,
\item a coideal, i.e.\ $\epsilon(\HIdeal)=0$ and $\Delta(\HIdeal) \subseteq \HIdeal\otimes \Hopf + \Hopf \otimes \HIdeal$ and
\item stable under the antipode, i.e.\ $S(\HIdeal)\subseteq \HIdeal$.
\end{enumerate}
Let $\Hopf$ be a graded Hopf algebra. 
Then we call $\HIdeal$ \emph{homogeneous} if for all $c \in \HIdeal$ with $c = \sum_{i=1}^n c_i$ and each $c_i$ homogeneous we have $c_i \in \HIdeal$ for $1 \leq i \leq n$.
\end{defn}

\begin{setup}[Quotient Hopf algebra and the annihilator of an ideal]
 Let $\Hopf$ be a Hopf algebra and let $\HIdeal\subseteq\Hopf$ be a Hopf ideal. 
 \begin{enumerate}
  \item The quotient vector space $\Hopf/\HIdeal$ carries a natural Hopf algebra structure (see \cite[Theorem 4.3.1.]{MR0252485}). 
  This structure turns the canonical quotient map $q \colon \Hopf \rightarrow \Hopf / \HIdeal$ into a morphism of Hopf algebras.
  \item Let $\lcB$ be a locally convex algebra. Then the algebra $\Hom_{\K}(\Hopf/\HIdeal,B)$ is canonically isomorphic to the \emph{annihilator of $\HIdeal$}:
 \[
  \Ann(\HIdeal, \lcB)= \setm{\phi \in \Hom_\K(\Hopf,\lcB)}{\phi(\HIdeal)=0_\lcB}
 \]
 which is a closed unital subalgebra of $\Hom_\K(\Hopf,\lcB)$.
 \end{enumerate}
 If $\Hopf$ is graded and the ideal $\HIdeal$ is homogeneous then the grading of $\Hopf$ induces a natural grading on the quotient $\Hopf/\HIdeal$. However, like the example of the universal enveloping algebra as a quotient of the tensor algebra (see Examples \ref{setup: tensor algebra} and \ref{setup: universal enveloping algebra}) shows, there are interesting ideals which occur naturally but are not homogeneous.
\end{setup}

\begin{lem}
 Let $\HIdeal$ be a Hopf ideal of the Hopf algebra $\Hopf$ with quotient mapping $\smfunc{q}{\Hopf}{\Hopf/\HIdeal}$.
 Let $\lcB$ be a commutative locally convex algebra. Then $\AnnGroup$ is a closed subgroup of the topological group $\Char{\Hopf}{\lcB}$.
 Furthermore it is isomorphic as a topological group to $\Char{\Hopf/\HIdeal}{\lcB}$ via the following isomorphism:
 \[
    \func{q_*}{ \Char{\Hopf / \HIdeal}{\lcB}   }{  \AnnGroup  }{\phi}{\phi \circ q.}
 \]
\end{lem}
\begin{proof}
 We first prove that $\AnnGroup$ is a closed subgroup.
 It is stable under the group product and contains the unit because $\Ann(\HIdeal, \lcB) $ is a unital subalgebra.
 To see that it is stable under inversion recall from Lemma \ref{lem: char:mult} that inversion in $\Char{\Hopf}{\lcB}$ is given by precomposition with the antipode. 
 Hence for $\phi \in \AnnGroup$ we find $\phi^{-1}(\HIdeal) = \phi \circ S (\HIdeal) \subseteq \phi (\HIdeal) =0$.
 Finally, $\AnnGroup$ is closed as a subset of $\Char{\Hopf}{\lcB}$ because $\Ann(\HIdeal, \lcB) $ is closed in $\lcA$, and $\Char{\Hopf}{\lcB}$ carries the subset topology.

 The map $\smfunc{q_*}{ \Char{\Hopf / \HIdeal}{\lcB}   }{  \AnnGroup  }$ is clearly an isomorphism of groups. Continuity of $q_*$ and $(q_*)^{-1}$ follows from the fact that we use pointwise convergence on all spaces.
\end{proof}

\begin{thm}\label{thm: ann:sbgp}
 Let $\Hopf$ be a graded connected Hopf algebra and $\HIdeal\subseteq \Hopf$ be a (not necessarily homogeneous) Hopf ideal.
 Furthermore, we fix a commutative locally convex algebra $\lcB$.
 Then
 \begin{itemize}
  \item[\normalfont (i)]   $\Char{\Hopf/\HIdeal}{B} \cong \AnnGroup \subseteq \Char{\Hopf}{\lcB}$ is a closed Lie subgroup, and even an exponential BCH--Lie group.
  \item[\normalfont (ii)]  $\AnnLieAlgebra \subseteq \InfChar{\Hopf}{\lcB}$ is a closed Lie subalgebra, and a BCH--Lie algebra.
  \item[\normalfont (iii)] The map $\exp$ restricts to a global $\K$-analytic diffeomorphism \\ $\AnnLieAlgebra \to \AnnGroup$.
  \item[\normalfont (iv)]  If $\HIdeal$ is homogeneous then the Lie group structure on $\AnnGroup$ agrees with the one already obtained on $\Char{\Hopf/\HIdeal}{\lcB}$.
 \end{itemize}
\end{thm}

\begin{proof}
\begin{enumerate}
 \item[(ii)] 
	$\AnnLieAlgebra$ is a Lie subalgebra because $[\phi,\psi] = \phi \star \psi - \psi \star \phi$ and $\Ann(\HIdeal, \lcB) $ is stable under convolution.
	It is closed because $\Ann(\HIdeal, \lcB) $ is closed in $\lcA$.
	As a closed Lie subalgebra of a BCH--Lie algebra, $\AnnLieAlgebra$ is also a BCH--Lie algebra.
\item[(iii)]
	From Theorem \ref{thm: Char:Lie} we deduce that it suffices to prove that $\exp$ restricts to a bijection $\AnnLieAlgebra \to \AnnGroup$.

	Recall from Theorem \ref{thm: Char:Lie} that the Lie group exponential map of $\Char{\Hopf}{\HIdeal}$ is a global diffeomorphism which is given on $\InfChar{\Hopf}{\lcB}\subseteq\IdealA$ by a convergent power series. 
	Hence $\exp$ maps elements in a closed unital subalgebra into the subalgebra, i.e.~$\exp(\phi)\in \Ann(\HIdeal,\lcB)$ for each $\phi\in\Ann(\HIdeal,\lcB)\cap \IdealA$ and thus
        \[
         \exp (\AnnLieAlgebra) \subseteq \AnnGroup.
        \]
	The logarithm $\log$ on $\Char{\Hopf}{\lcB}\subseteq(1_\lcA + \IdealA)$ is also given by a power series which converges on $\AnnGroup$ and by the same argument, we obtain:
        \[
         \log (\AnnGroup)    \subseteq \AnnLieAlgebra.
        \]
	In conclusion, $\exp$ restricts to a bijection $\AnnLieAlgebra \to \AnnGroup$ as desired.
\item[(i)]
	It now follows from (ii),(iii), and \cite[Theorem IV.3.3]{MR2261066} that $\AnnGroup$ is a closed Lie subgroup, and an exponential BCH--Lie group. 
\item[(iv)]
	We have already seen that $\Char{\Hopf / \HIdeal}{\lcB}$ and $\AnnGroup$ are isomorphic as topological groups and 
	that $\Char{\Hopf / \HIdeal}{\lcB}$ and $\AnnGroup$ are Baker--Campbell--Hausdorff--Lie groups.
	The Automatic Smoothness Theorem \cite[Theorem IV.1.18]{MR2261066} thus implies that the isomorphisms are in fact ($\K$-analytic) morphisms of Lie groups. \qedhere
\end{enumerate}
\end{proof}
 
 Note that when $\HIdeal$ is homogeneous and $\lcB$ is Mackey complete, it follows from Theorem \ref{thm: ann:sbgp}(iv) that the Lie subgroup $\AnnGroup$ is again a regular Lie group. 
 Namely, we derive together with Theorem \ref{thm: Char_regular} the following regularity result: 
 
 \begin{cor} \label{cor: ann:reg}
  Let $\Hopf$ be a graded connected Hopf algebra and $\lcB$ be a commutative and Mackey complete locally convex algebra. 
  Furthermore, let $\HIdeal \subseteq \Hopf$ be a homogeneous Hopf ideal, then the Lie subgroup $\AnnGroup$ is a $C^1$-regular Lie group.
  If $\lcB$ is in addition sequentially complete, then $\AnnGroup$ is $C^0$-regular.  
 \end{cor}

 Note that we have not established the regularity condition of $\AnnGroup$ for non-homogeneous $\HIdeal$. 
 
 \begin{problem}
  Let $\Hopf$ be a graded and connected Hopf algebra and $\lcB$ be a commutative and Mackey complete locally convex algebra. 
  Is the Lie group $\AnnGroup$ $C^k$-regular (with $k \in \N_0 \cup \{ \infty \}$) if $\HIdeal$ is a non-homogeneous Hopf Ideal?
  Again it suffices to prove that $\AnnGroup$ is a semiregular Lie subgroup of $\Char{\Hopf}{\lcB}$.
  However, the idea from the proof of Theorem \ref{thm: Char_regular} seems to carry over only to the situation where $\HIdeal$ is homogeneous. 
 \end{problem}

In the special case that $\lcB$ is a weakly complete algebra (e.g.\ a finite-dimensional algebra) we deduce from Remark \ref{rem: lotsastuff} (c) the following corollary. 
 
 \begin{cor}
  Let $\Hopf$ be a connected graded Hopf algebra and $\lcB$ be a commutative and weakly complete locally convex algebra. 
  Furthermore, let $\HIdeal \subseteq \Hopf$ be a Hopf ideal. 
  Then the Lie subgroup $\AnnGroup$ is regular.
 \end{cor}

 \section{(Counter-)examples for Lie groups arising as Hopf algebra characters} \label{section: Examples}
 
 In this section we give several examples for Lie groups arising from the construction in the last section. 
 In the literature many examples for graded and connected Hopf algebras are studied (we refer the reader to \cite{MR2290769} and the references and examples therein).
 In particular, the so called combinatorial Hopf algebras provide a main class of examples for graded and connected Hopf algebras (see \cite{MR2732058} for an overview).
 A prime example for a combinatorial Hopf algebra is the famous Connes--Kreimer Hopf algebra of rooted trees.
 Its character group corresponds to the Butcher group from numerical analysis and we discuss it as our main example below.
 Furthermore, we discuss several (counter-)examples to statements in Theorem \ref{thm: Char:Lie} for characters of Hopf algebras which are \emph{not} graded.

 \subsection*{Tensor algebras and universal enveloping algebras}
  \begin{setup}[Tensor algebra]															\label{setup: tensor algebra}
   Consider an abstract vector space $\VV$. Then the tensor algebra
   \[
    T(\VV)\coloneq \bigoplus_{n=0}^\infty \VV^{\otimes n} \text{ with } \VV^{\otimes n} \coloneq \underbrace{\VV \otimes \cdots\otimes \VV }_{n}
   \]
   has a natural structure of a graded connected Hopf algebra $(T(\VV),\otimes,u,\Delta,\epsilon,S)$ with
   \[
    \Delta(v)=1\otimes v + v\otimes 1 \hbox{ and } S(v)=-v \hbox{ for } v\in\VV.
   \]
   By Theorem \ref{thm: Char:Lie} the character group of $T(\VV)$ is a BCH--Lie group. 
   This group can be described explicitly.

   Every linear functional on $\VV$ has a unique extension to a character of the Hopf algebra $T(\VV)$, yielding a bijection to the algebraic dual $\VV^*$:
   \[
    \func{\Phi}{\Char{T(\VV)}{\K}}{\VV^*}{\phi}{\phi|_\VV.}
   \]
   We claim that $\Phi$ is a group isomorphism (where we view $\VV^*$ as a group with respect to its additive structure). Let $v\in\VV$ and $\phi,\psi\in\Char{T(\VV)}{\K}$ be given, then
   \begin{align*}
    (\phi \star \psi) (v) &= m_\K \circ (\phi\otimes \psi)(\Delta (v) ) = m_\K (\phi\otimes\psi)(v\otimes1+1\otimes v) \\
			  &= \phi(v)\psi(1)+\phi(1)\psi(v)=\phi(v)+\psi(v).
   \end{align*}
   Thus, the group $(\Char{T(\VV)}{\K},\star)$ is isomorphic to the additive group $(\VV^*,+)$.
   As $\VV^*$ is endowed with the weak*-topology, it is easy to check that $\Phi$ and $\Phi^{-1}$ are both continuous, hence $\Phi$ is also an isomorphism of topological groups. Since both Lie groups are known to be BCH--Lie groups, the the Automatic Smoothness Theorem \cite[Theorem IV.1.18.]{MR2261066} guarantees that $\Phi$ is also an analytic diffeomorphism, hence an isomorphism in the category of analytic Lie groups.
  \end{setup}

  \begin{setup}[Universal enveloping algebra]															\label{setup: universal enveloping algebra}
   The \emph{universal enveloping algebra} $\cU (\g)$ of an abstract non-graded Lie algebra $\g$ can be constructed as a quotient of the connected graded Hopf algebra $T(\g)$ and hence, its character group is a Lie group by Theorem \ref{thm: ann:sbgp}. Note that we cannot use Theorem \ref{thm: Char:Lie} directly since in general $\cU(\g)$ does not possess a natural connected grading (the grading of the tensor algebra induces only a filtration on $\cU (\g)$, see \cite[Theorem V.2.5]{MR1321145}). 
   If $\g$ is abelian, the universal enveloping algebra $\cU (\g)$ coincides with the symmetric algebra $S(\g)$ (cf.\ \cite[V.2 Example 1]{MR1321145}). 
   It is possible to give an explicit description of the group $\Char{\cU(\g)}{\K}$:
   Every character of $\phi\in\Char{\cU(\g)}{\K}$ corresponds to a Lie algebra homomorphism $\smfunc{\phi|_\g}{\g}{\K}$ which in turn factors naturally through a linear map $\smfunc{\phi^\sim}{\g/(\g')}{\K}$, yielding a bijection
   \[
    \func{\Phi}{\Char{\cU(\g)}{\K}}{\bigl((\g/\g')^*,+\bigr)}{\phi}{\bigl( \phi^\sim\colon v+\g' \mapsto \phi(v) \bigr).}
   \]
   Like in the case of the tensor algebra and the symmetric algebra, one easily verifies that this is an isomorphism of topological groups and 
   since both Lie groups $\Char{\cU(\g)}{\K}$ and $\bigl((\g/\g')^*,+\bigr)$ are BCH--Lie groups, we use again the Automatic Smoothness Theorem \cite[Theorem IV.1.18.]{MR2261066} to see that they are also isomorphic as analytic Lie groups.

   In particular, this shows that the character group of $\cU(\g)$ only sees the abelian part of $\g$ and is therefore not very useful for studying the Lie algebra $\g$.
  \end{setup}
  
  \begin{rem}[Universal enveloping algebra of a graded Lie algebra] 
   We remark that there is a notion of a \emph{graded Lie algebra} which differs from the usual notion of a Lie algebra. Such graded Lie algebras also admit a universal enveloping algebra which inherits a grading from the Lie algebra grading (see \cite[5.]{MR0174052}). 
   Then, the universal enveloping algebra becomes a graded (in general not \emph{connected}) Hopf algebra. 
  \end{rem}

 \subsection*{Characters of the Hopf algebra of rooted trees}
 
  We examine the Hopf algebra of rooted trees which arises naturally in numerical analysis, renormalisation of quantum field theories and non-commutative geometry (see \cite{Brouder-04-BIT} for a survey).
  To construct the Hopf algebra, recall some notation first.
 
 \begin{nota}
 \begin{enumerate}
  \item A \emph{rooted tree} is a connected \emph{finite} graph without cycles with a distinguished node called the \emph{root}. 
  We identify rooted trees if they are graph isomorphic via a root preserving isomorphism.
  
  Let $\RT$ be \emph{the set of all rooted trees} and write $\RT_0 \coloneq \RT \cup \{\emptyset\}$ where $\emptyset$ denotes the empty tree.
  The \emph{order} $|\tau|$ of a tree $\tau \in \RT_0$ is its number of vertices. 
  \item An \emph{ordered subtree}\footnote{The term ``ordered'' refers to that the subtree remembers from which part of the tree it was cut.} of $\tau \in \RT_0$ is a subset $s$ of all vertices of $\tau$ which satisfies 
    \begin{compactitem}
     \item[(i)] $s$ is connected by edges of the tree $\tau$,
     \item[(ii)] if $s$ is non-empty, it contains the root of $\tau$.
    \end{compactitem}
   The set of all ordered subtrees of $\tau$ is denoted by $\OST (\tau)$.
   Associated to an ordered subtree $s \in \OST (\tau)$ are the following objects:
    \begin{compactitem}
     \item A forest (collection of rooted trees) denoted as $\tau \setminus s$. 
     The forest $\tau \setminus s$ is obtained by removing the subtree $s$ together with its adjacent edges from $\tau$.
     We denote by $\# (\tau \setminus s)$ the number of trees in the forest $\tau \setminus s$.
     \item $s_\tau$, the rooted tree given by vertices of $s$ with root and edges induced by that of the tree $\tau$.
    \end{compactitem}
 \end{enumerate}
\end{nota}

\begin{nota}
 A \emph{partition} $p$ of a tree $\tau \in \RT_0$ is a subset of edges of the tree.
 We denote by $\cP (\tau)$ the set of all partitions of $\tau$ (including the empty partition). 
  Associated to a partition $p \in \cP (\tau)$ are the following objects 
    \begin{compactitem}
     \item A forest $\tau \setminus p$ which is defined as the forest that remains when the edges of $p$ are removed from the tree $\tau$.
     Write $\#(\tau\setminus p)$ for the number of trees in $\tau \setminus p$.
     \item The \emph{skeleton} $p_\tau$, is the tree obtained by contracting each tree of $\tau \setminus p$ to a single vertex and by re-establishing the edges of $p$.
    \end{compactitem}
 \end{nota}
 
 \begin{ex}[The Connes--Kreimer Hopf algebra of rooted trees \cite{CK98}]\label{ex: RT}
  Consider the algebra $\Hopf^{\K}_{CK} \coloneq \K [\cT]$ of polynomials which is generated by the trees in $\RT$. 
  We denote the structure maps of this algebra by $m$ (multiplication) and $u$ (unit).
  Indeed $\Hopf^\K_{CK}$ becomes a bialgebra with the coproduct  
    \begin{displaymath}
     \Delta \colon \Hopf^\K_{CK} \rightarrow \Hopf^\K_{CK} \otimes \Hopf^\K_{CK} , \tau \mapsto \sum_{s \in \OST (\tau)} (\tau \setminus s) \otimes s_\tau 
    \end{displaymath}
 and the counit $\epsilon \colon \Hopf^\K_{CK} \rightarrow \K$ defined via $\epsilon (1_{\Hopf^\K_{CK}}) = 1$ and $\epsilon (\tau) = 0$ for all $\tau \in \RT$.
 Furthermore, one can define an antipode $S$ via 
  \begin{displaymath}
   S \colon \Hopf^\K_{CK} \rightarrow \Hopf^\K_{CK} , \tau \mapsto \sum_{p \in \cP (\tau)} (-1)^{|p_\tau|} (\tau \setminus p)
  \end{displaymath}
 such that $\Hopf^\K_{CK} = (\Hopf^\K_{CK} ,m,u, \Delta,\epsilon,S)$ is a $\K$-Hopf algebra (see \cite[5.1]{CHV2010} for more details and references).
 
 Furthermore, the Hopf algebra $\Hopf^\K_{CK}$ is graded as a Hopf algebra by the \emph{number of nodes grading}: 
 For each $n \in \N_0$ we define the $n$th degree via
  \begin{displaymath}
  \text{For } \tau_i \in \RT, 1 \leq i \leq k, k \in \N_0 \quad \quad \tau_1 \cdot \tau_2 \cdot \ldots \cdot \tau_k  \in (\Hopf^\K_{CK})_n \text{ iff } \sum_{r= 1}^k |\tau_k| = n 
  \end{displaymath}
 Clearly, $\Hopf^\K_{CK}$ is connected with respect to the number of nodes grading and we identify $(\Hopf_{CK}^\K)_0$ with $\K \emptyset$. 
 Thus we can apply Theorem \ref{thm: Char:Lie} for every commutative CIA $\lcB$ to see that the $\lcB$ valued characters $\Char{\Hopf^\K_{CK}}{\lcB}$ form an exponential BCH--Lie group.
 \end{ex}
 
 It turns out that the $\K$-valued characters of the Connes--Kreimer Hopf algebra $\Hopf^\K_{CK}$ can be identified with the Butcher group from numerical analysis.
 
 \begin{ex}[The Butcher Group]\label{ex: Butcher}
  Let us recall the definition of the ($\K$-)Butcher group.
  As a set the ($\K$-)Butcher group is the set of tree maps 
    \begin{displaymath}
     G_{\text{TM}}^\K \coloneq \{a \colon \RT_0 \rightarrow \K \mid a (\emptyset) = 1 \}
    \end{displaymath}
 together with the group operation 
  \begin{displaymath}
    a\cdot b (\tau) \coloneq \sum_{s \in \OST (\tau)} b(s_\tau) a(\tau \setminus s) \quad \text{with} \quad a(\tau \setminus s) \coloneq \prod_{\theta \in \tau \setminus s} a (\theta).
  \end{displaymath}
 In \cite{BS14} we have constructed a Lie group structure for the ($\K$-)Butcher group as follows: 
 Choose an enumeration $\N \rightarrow \RT, n \mapsto \tau_n$ of the rooted trees.
 Then the global parametrisation 
 \begin{displaymath}
  \varphi^\K \colon \K^\N \rightarrow  G_{\text{TM}}^\K, (a_n)_{n \in \N} \mapsto \left(\tau \mapsto \begin{cases}
                                                                                                      1 &\text{if } \tau = \emptyset \\
												      a_n & \text{if } \tau = \tau_n
                                                                                                     \end{cases} \right)
\end{displaymath}
 turns $G_{\text{TM}}^\K$ into a BCH--Lie group modelled on the \Frechet space $\K^\N$. 
 \end{ex}

 Note that the group operation of the Butcher group is closely related to the coproduct of the Hopf algebra of rooted trees.
 Indeed the obvious morphism 
  \begin{displaymath}
   \Phi \colon \Char{\Hopf^\K_{CK}}{\K} \rightarrow G_{\text{TM}}^\K , \varphi \mapsto (\tau \mapsto \varphi(\tau))
  \end{displaymath}
 is an isomorphism of (abstract) groups (see also \cite[Eq.\ 38]{CHV2010}). 
 Moreover, it turns out that $\Phi$ is an isomorphism of Lie groups if we endow these groups with the Lie group structures discussed in Example \ref{ex: RT} and Example \ref{ex: Butcher}.
 
 \begin{lem}
 \label{lem: butcheriso}
  The group isomorphism $\Phi \colon \Char{\Hopf^\K_{CK}}{\K} \rightarrow G_{\text{TM}}^\K $ is an isomorphism of $\K$-analytic Lie groups. 
 \end{lem}

  \begin{proof}
   We already know that $\Phi$ is an isomorphism of abstract groups whose inverse is given by $\Phi^{-1} (a) = \varphi_a$ where $\varphi_a$ is the algebra homomorphism defined via
    \begin{displaymath}
     \varphi_a (1_{\Hopf^\K_{CK}}) = 1 \text{ and } \varphi_{a} (\tau) = a(\tau)
    \end{displaymath}
   Recall from \ref{thm: glockner_CIA_BCH} that the Lie group $\lcA^\times \coloneq \Hom_K (\Hopf_{CK}^\K , \K)^\times$ carries the subspace topology with respect to the topology of pointwise convergence on $\Hom_K (\Hopf_{CK}^\K , \K)$. 
   Since by Theorem \ref{thm: Char:Lie} the character group $\Char{\Hopf^\K_{CK}}{\K}$ is a closed subgroup of $\lcA^\times$ its topology is the subspace topology with respect to the topology of pointwise convergence on $\Hom_K (\Hopf_{CK}^\K , \K)$.
   Furthermore, the topology on $G_{\text{TM}}^\K$ is the identification topology with respect to the parametrisation $\varphi^\K$ and the model space $\K^\N$ is endowed with the product topology. 
   Hence a straight forward computation shows that $\Phi$ and $\Phi^{-1}$ are continuous, i.e.\ they are isomorphisms of topological groups.
   Since both $\Char{\Hopf^\K_{CK}}{\K}$ and $G_{\text{TM}}^\K$ are BCH--Lie groups, the Automatic Smoothness Theorem \cite[Theorem IV.1.18]{MR2261066} asserts that $\Phi$ and $\Phi^{-1}$ are smooth (even real analytic).
   Thus $\Phi$ is an isomorphism of ($\K$-analytic) Lie groups.
  \end{proof}
 
 So far we have seen that our Theorem \ref{thm: Char:Lie} generalises the construction of the Lie group structure of the Butcher group from \cite{BS14}. 
 In loc.cit. we have also endowed the subgroup of symplectic tree maps with a Lie group structure. 
 This can be seen as a special case of the construction given in Theorem \ref{thm: ann:sbgp} as the next example shows. 
 Thus the results from \cite{BS14} are completely subsumed in the more general framework developed in the present paper.
 
 \begin{ex}[The subgroup of symplectic tree maps]						\label{ex: Symplectic tree maps}
  In \cite[Theorem 5.8]{BS14}, it was shown that the subgroup of symplectic tree maps $S^\K_{\text{TM}}\subseteq  G_{\text{TM}}^\K$ is a closed Lie subgroup of the Butcher group and that the subgroup is itself an exponential Baker--Campbell--Hausdorff Lie group.
  
  The symplectic tree maps are defined as those $a\in  G_{\text{TM}}^\K$ such that 
  \begin{displaymath}
   a(\tau \circ \upsilon) + a(\upsilon \circ \tau) = a(\tau)a(\upsilon) \text{ for all } \tau, \upsilon \in \RT,
  \end{displaymath}
 where $\tau \circ \upsilon$ denotes the rooted tree obtained by connecting $\tau$ and $\upsilon$ with an edge between the roots of $\tau$ and $\upsilon$, and where the root of $\tau$ is the root of $\tau \circ \upsilon$ \footnote{This is known as the \emph{Butcher product} and should not be confused with the product in the Butcher group (cf.\ \cite[Remark 5.1]{BS14}).}.
  
  To cast \cite[Theorem 5.8]{BS14} in the context of Theorem \ref{thm: ann:sbgp}, let $\HIdeal \subseteq \Hopf_{CK}$ be the algebra ideal generated by the elements $\{\tau \circ \upsilon + \upsilon \circ \tau -\tau\upsilon\}_{\tau, \upsilon \in \RT}$.
  Note that by definition of the Butcher product we have $|\tau \circ \upsilon | = | \upsilon \circ \tau| = |\tau\upsilon|$. 
  Hence the generating elements of $\HIdeal$ are homogeneous elements with respect to the number of nodes grading (see Example \ref{ex: RT}).
  Consequently $\HIdeal$ is a homogeneous (algebra) ideal.
  It is possible to show that $\HIdeal$ is also a co-ideal and stable under the antipode.
  
  If $a\in S^\K_{\text{TM}}$, then $\varphi_a = \Phi^{-1}(a)\in \Ann(\HIdeal, \lcB) \cap \Char{\Hopf}{\K}$, since $\varphi_a$ is an algebra morphism and zero on the generators of $\HIdeal$.
  The inverse implication also holds.
  Therefore, the restriction of $\Phi$ is a bijection between $S^\K_{\text{TM}}$ and $ \Ann(\HIdeal, \lcB) \cap \Char{\Hopf}{\K}$.
  By Theorem \ref{thm: ann:sbgp}, $\Ann(\HIdeal, \lcB) \cap \Char{\Hopf}{\K}\subseteq \Char{\Hopf}{\K}$ is a closed Lie subgroup and an exponential BCH--Lie group.
  Using Lemma \ref{lem: butcheriso}, we can show that this structure is isomorphic to the Lie group structure $S^\K_{\text{TM}} \subseteq  G_{\text{TM}}^\K$ constructed in \cite[Theorem 5.8]{BS14}.
  
  In addition to the Lie group structure on $S^\K_{\text{TM}}$ already constructed in \cite[Theorem 5.8]{BS14} we derive from Corollary \ref{cor: ann:reg} that the Lie group $S^\K_{\text{TM}}$ is $C^0$-regular.
 \end{ex}

 \subsection*{Characters of Hopf algebras without connected grading}
 
 For the rest of this section let us investigate the case of a Hopf algebra $\Hopf$ without a connected grading. 
 It will turn out that the results achieved for graded connected Hopf algebras (and quotients thereof) do not hold for Hopf algebras without grading.
 It should be noted however, that for scalar valued characters, we can show that they still form a so called pro-Lie group (see Theorem \ref{thm: character_group_is_a_pro_lie_group})
 \medskip
 
 Let $\Hopf$ be a Hopf algebra without a grading then the dual space $\lcA \coloneq \Hopf^*=\Hom_{\K}(\Hopf,\K)$ is still a locally convex algebra (see \ref{setup: maps on coalgebra}). 
 However, in general, neither its unit group will be an open subset, nor will the group of characters $\InfChar{\Hopf}{\K}$ be a Lie group modelled on a locally convex space.
 We give two examples for this behaviour.

 \begin{ex}														\label{ex: group algebra}
  Let $\Gamma$ be an abstract group. 
  Then the group algebra $\K\Gamma$ carries the structure of a cocommutative Hopf algebra by \cite[III.3 Example 2]{MR1321145}.
  The algebraic dual $\lcA\coloneq (\K\Gamma)^*$ is isomorphic to the direct product algebra $\K^\Gamma$ consisting of all functions on the group with pointwise multiplication. 
  Its unit group $(\K\setminus\smset{0})^\Gamma$ is a topological group (as a direct product of the topological group $\K\setminus\smset{0}$ with itself). 
  However, in general, it will not be open:
  
  \begin{enumerate}
   \item Let $\Gamma$ be an infinite group, then the unit group $\lcA^\times = (\K\setminus\smset{0})^\Gamma$ is not open in $\K^\Gamma$. 
    Hence the methods used to prove Theorem \ref{thm: Char:Lie} do not generalise to this situation.
  \end{enumerate}
  Furthermore, the construction of the algebra $\K\Gamma$ in \cite[III.2 Example 2]{MR1321145} implies that a linear map $\smfunc{\phi}{\K\Gamma}{\K}$ is a character if and only if the map $\smfunc{\phi|_\Gamma}{\Gamma}{\K^\times}$ is a group homomorphism.
   
   Thus, the group $\Char{\K\Gamma}{\K}$ is (as a topological group) isomorphic to the group of group homomorphisms from $\Gamma$ to $\K^\times$ with the pointwise topology.
  \begin{enumerate}
   \item[(b)]  Let $\Gamma=(\Z^{(I)},+)$ be a free abelian of countable infinite rank.
   Then it is easy to see that $\Char{\K\Gamma}{\K}$ is topologically isomorphic to the infinite product $(\K^\times)^I$. 
   This topological group is not locally contractible, hence it can not be (locally) homeomorphic to a topological vector space and thus cannot carry a locally convex manifold structure. 
   In particular, Theorem \ref{thm: Char:Lie} does not generalise to the character group of $\K\Z^{(I)}$.
  \end{enumerate}
 \end{ex}

 If a non-graded Hopf algebra $\Hopf$ is finite dimensional, then $\lcA=\Hopf^*$ is a finite dimensional algebra and hence it is automatically a CIA. The group of characters will then be a (finite dimensional) Lie group.
 However, by \cite[2.2]{MR547117} the characters are linearly independent whence this Lie group will always be finite.
 Hence there cannot be a bijection between the character group and the Lie algebra of infinitesimal characters (which in this case will be $0$-dimensional). 
 This shows that even when $\Char{\Hopf}{\K}$ is a Lie group it may fail to be exponential. We consider a concrete example of this behaviour:
 
 \begin{ex}														\label{ex: functions on a finite group}
  Take a \emph{finite} non trivial group $\Gamma$ and consider the finite dimensional algebra $\Hopf\coloneq\K^\Gamma$ of functions on the group with values in $\K$ together with the pointwise operations. 
  There is a suitable coalgebra structure and antipode (see \cite[Example 1.5.2]{MR1381692}) which turns algebra $\K^\Gamma$ into a Hopf algebra. 
  Furthermore, we can identify its dual with the group algebra $\K\Gamma$ of $\Gamma$ (as \cite[III. Example 3]{MR1321145} shows). 
  
  With this identification we can identify $\Char{\Hopf}{\K}$ with the group $\Gamma$ (with the discrete topology). 
  Obviously, there is no bijection between the group of characters and the $\K$-Lie algebra of infinitesimal characters (which in this case is trivial). 
 \end{ex}

 \section{Character groups as pro-Lie groups}											\label{section: pro-Lie}
  In this section, we show that the group of characters of an abstract Hopf algebra (graded or not) can always be considered as a projective limit of finite dimensional Lie groups, i.e.~the group of characters is a \emph{pro-Lie group}. The range space $B$ has to the ground field $\K\in\smset{\R,\C}$ or a commutative weakly complete algebra (see Lemma \ref{lem: fundamental_lemma_of_weakly_complete_algebras}).

 The category of pro-Lie groups admits a very powerful structure theory which is similar to the theory of finite dimensional Lie groups (see \cite{MR2337107}). 
 This should provide Lie theoretic tools to work with character groups, even in the examples where the methods of locally convex Lie groups do not apply (like Examples \ref{ex: group algebra} and \ref{ex: functions on a finite group}). 
 It should be noted, however, that the concept of a pro-Lie group is of purely topological nature and involves no differential calculus. 
See \cite{MR2475971} for an article dedicated to the problem which pro-Lie groups do admit a locally convex differential structure and which do not.

  \begin{defn}[pro-Lie group]													\label{defn: pro_lie_group}
   A topological group $G$ is called \emph{pro-Lie group} if one of the following equivalent conditions holds:
   \begin{itemize}
    \item [\textup{(a)}] $G$ is topologically isomorphic to a closed subgroup of a product of finite dimensional (real) Lie groups.
    \item [\textup{(b)}] $G$ is the projective limit of a directed system of finite dimensional (real) Lie groups (taken in the category of topological groups)
    \item [\textup{(c)}] $G$ is complete and each identity neighbourhood contains a closed normal subgroup $N$ such that $G/N$ is a finite dimensional (real) Lie group.
   \end{itemize}
  \end{defn}
  The fact that these conditions are equivalent is surprisingly complicated to show and can be found in \cite{MR2377913} or in \cite[Theorem 3.39]{MR2337107}. The class of pro-Lie groups contains all compact groups (see e.g.~\cite[Corollary 2.29]{MR3114697}) and all connected locally compact groups (Yamabe's Theorem, see \cite{MR0054613}). However, this does not imply that all pro-Lie groups are locally compact. In fact, the pro-Lie groups constructed in this paper will almost never be locally compact.

  In absence of a differential structure we cannot define a Lie algebra as a tangent space.
  However, it is still possible to define a Lie functor.

  \begin{setup}[The pro-Lie algebra of a pro-Lie group]
   Let $G$ be a pro-Lie group. 
   Consider the space $\cL(G)$ of all continuous $G$-valued one-parameter subgroups, endowed with the compact-open topology. 
   It is the projective limit of finite dimensional Lie algebras and hence, carries a natural structure of a locally convex topological Lie algebra over $\R$ (see \cite[Definition 2.11]{MR2337107}). 
   As a topological vector space, $\cL(G)$ is weakly complete, i.e.~isomorphic to $\R^I$ for an index set $I$ (see also Definition \ref{def: weakly_complete_space}).

   This assignment is functorial: Assign to a morphism of pro-Lie groups, i.e.~a continuous group homomorphism $\smfunc{\phi}{G}{H}$, a morphism of topological real Lie algebras $\func{\cL(\phi)}{\cL(G)}{\cL(H)}{\gamma}{\gamma\circ\phi}$. 
   For more information on pro-Lie groups, pro-Lie algebras and the pro-Lie functor, see \cite[Chapter 3]{MR2337107}.
  \end{setup}

  Many pro-Lie groups arise as groups of invertible elements of topological algebras:
  \begin{lem}[Fundamental lemma of weakly complete algebras]									\label{lem: fundamental_lemma_of_weakly_complete_algebras}
   For a topological $\K$-algebra $\lcA$, the following are equivalent:
   \begin{itemize}
    \item [\textup{(a)}] The underlying topological vector space $\lcA$ is (forgetting the multiplicative structure) weakly complete, i.e.~isomorphic to $\K^I$ for an index set $I$.
    \item [\textup{(b)}] There is an abstract $\K$-coalgebra $(\CoAlg,\Delta_\CoAlg,\epsilon_\CoAlg)$ such that $\lcA \cong (\Hom(\CoAlg,\K),\star)$.
    \item [\textup{(c)}] The topological algebra $\lcA$ is the projective limit of a directed system of \emph{finite dimensional} $\K$-algebras (taken in the category of topological $\K$-algebras)
   \end{itemize}
   A topological algebra with these properties is called \emph{weakly complete algebra}.
  \end{lem}
  \begin{proof}
  \begin{itemize}
   \item[(a)$\Rightarrow$(b)] The category of weakly complete topological vector spaces over $\K$ and the category of abstract $\K$-vector spaces are dual.
    This implies that $\lcA=\K^I$ is the algebraic dual space of the vector space $\CoAlg\coloneq\K^{(I)}$ of finite supported functions.
    The continuous multiplication $\smfunc{\mu_{\lcA}}{\lcA\times\lcA}{\lcA}$ dualises to an abstract comultiplication $\smfunc{\Delta_\CoAlg}{\CoAlg}{\CoAlg\otimes\CoAlg}$. 
    (see \cite[Theorem 4.4]{MR2035107} or Appendix \ref{app: weakly complete} for details of this duality.)
   \item[ (b)$\Rightarrow$(c)]  This is a direct consequence of the \emph{Fundamental Theorem of Coalgebras} (\hspace{-1pt}\cite[Theorem 4.12]{MR2035107}) stating that $\CoAlg$ is the directed union of finite dimensional coalgebras.
    Dualising this, yields a projective limit of topological algebras.
   \item[(c)$\Rightarrow$(a)]  The projective limit of finite dimensional $\K$-vector spaces is always topologically isomorphic to $\K^I$. (see Appendix \ref{app: weakly complete}) \qedhere
  \end{itemize}
 \end{proof}

  \begin{prop}[The group of units of a weakly complete algebra is a pro-Lie group]						\label{prop: group_of_units_pro_lie_group}
   Let $\lcA$ be a weakly complete $\K$-algebra as in Lemma \ref{lem: fundamental_lemma_of_weakly_complete_algebras}. Then the group of units $\lcA^\times$ is a pro-Lie group. Its pro-Lie algebra $\cL(\lcA)$ is (as a \emph{real} Lie algebra) canonically isomorphic to $(\lcA,\LB)$ via the isomorphism
   \[
    \nnfunc{\lcA}{\cL(\lcA^\times)}{x}{\gamma_x:(t\mapsto \exp(tx)),}
   \]
   where $\smfunc{\exp}{\lcA}{\lcA^\times}$ denotes the usual exponential series which converges on $\lcA$.
  \end{prop}
  \begin{proof}
   Let $\lcA = \lim_{\leftarrow} \lcA_\alpha$ with finite dimensional $\K$-algebras $\lcA_\alpha$ (by Lemma \ref{lem: fundamental_lemma_of_weakly_complete_algebras}). Then the unit group is given by
   \[
    \lcA^\times = \lim_{\leftarrow} \lcA_\alpha^\times
   \]
   in the category of topological groups.
    Each group $\lcA_\alpha^\times$ is a finite dimensional (linear) real Lie group. 
    Hence, $\lcA^\times$ is a pro-Lie group and in particular, inversion is continuous, which is not obvious for unit groups of topological algebras.

   The exponential series converges on each algebra $\lcA_\alpha$ and hence on the projective limit $\lcA$. The correspondence between continuous one-parameter subgroups $\gamma\in\cL(\lcA^\times)$ and elements in $\lcA$ holds in each $\lcA_\alpha$ and hence it holds on $\lcA$.
  \end{proof}

  \begin{rem}
   As Example \ref{ex: group algebra} shows, this group $\lcA^\times$ need not be an open subset of $\lcA$, nor will the exponential series be local homeomorphism around $0$.
  \end{rem}

  \begin{thm}[The character group of a Hopf algebra is a pro-Lie group]							\label{thm: character_group_is_a_pro_lie_group}
   Let $\Hopf$ be an abstract Hopf algebra and $\lcB$ be a commutative weakly complete $\K$-algebra (e.g.~$\lcB\coloneq\K$ or $\lcB=\K[[X]]$).
   Then the group of $\lcB$-valued characters $\Char{\Hopf}{\lcB}$ endowed with the topology of pointwise convergence is pro-Lie group. 
   
   Its pro-Lie algebra is isomorphic to the locally convex Lie algebra $\InfChar{\Hopf}{\lcB}$ of infinitesimal characters via the canonical isomorphism
   \[
    \nnfunc{\InfChar{\Hopf}{\lcB}}{\cL(\Char{\Hopf}{\lcB})  }{\phi}{(t\mapsto \exp(t\phi)),}
   \]
  \end{thm}
  \begin{rem}
   The pro-Lie algebra $\cL(G)$ of a pro-Lie group $G$ is a priori only a \emph{real} Lie algebra\footnote{This is due to the fact that the finite dimensional \emph{real} Lie groups form a full subcategory of the category of topological groups while the finite dimensional \emph{complex} Lie groups do not. In fact, there are infinitely many non-isomorphic complex Lie group structures (elliptic curves) on the torus $(\R/\Z)^2$, inducing the same real Lie group structure.}
   However, since we already know that $\InfChar{\Hopf}{\lcB}$ is a complex Lie algebra if $\K=\C$ (Lemma \ref{lem: inf:subalg}), we may use the isomorphism given in the theorem above to endow the real Lie algebra structure with a complex one.
  \end{rem}

  \begin{proof}[Proof of Theorem \ref{thm: character_group_is_a_pro_lie_group}]
   The space $\lcA\coloneq\Hom_{\K}(\Hopf,\lcB)$ is a topological algebra by \ref{setup: maps on coalgebra}. 
   The underlying topological vector space is isomorphic to $\lcB^I$ by Lemma \ref{lem: completeness of A}. 
   Thus, $\lcA$ is a weakly complete algebra since $\lcB$ is weakly complete.

   By Proposition \ref{prop: group_of_units_pro_lie_group}, we may conclude that $\lcA^\times$ is a pro-Lie group. The group $\Char{\Hopf}{\lcB}$ is a closed subgroup of this pro-Lie group by Lemma \ref{lem: char:mult}. 

   From part (a) of Definition \ref{defn: pro_lie_group} it follows that closed subgroups of pro-Lie groups are pro-Lie groups. Hence, $\Char{\Hopf}{\lcB}$ is a pro-Lie group.

   It remains to show that the pro-Lie algebra $\cL(\Char{\Hopf}{B})$ is isomorphic to $\InfChar{\Hopf}{B}$.

   Since $\Char{\Hopf}{\lcB}$ is a closed subgroup of $\lcA^\times$, every continuous $1$-parameter-subgroup $\gamma$ of $\Char{\Hopf}{\lcB}$ is also a $1$-parameter subgroup of $\lcA^\times$ and (by Proposition \ref{prop: group_of_units_pro_lie_group}) of the form
   \[
    \func{\gamma_\phi}{\R}{\lcA^\times}{t}{\exp(t \phi)}
   \]
   for a unique element $\phi\in\lcA$. It remains to show the following equivalence:
   \[
    \left( \forall t\in\R \colon \exp(t \phi)\in\Char{\Hopf}{\lcB} \right)\quad\iff\quad \phi\in\InfChar{\Hopf}{\lcB}.
   \]
   At the end of the proof of Lemma \ref{lem: exp:bij}, there is a chain of equivalences.
   While the equivalence of the first line with the second uses the bijectivity of the exponential function which does not hold in the pro-Lie setting, the equivalence of the second line with all following lines hold by Remark \ref{rem: salv:wc} also in this setting.
   Substituting $t\phi$ for $\phi$, we obtain the following equivalence:
   \newcommand{\td}{\diamond}
   \[
     \forall t\in\R \colon \exp(t \phi)\in\Char{\Hopf}{\lcB} \quad\iff\quad \exp_{\lcAt}(t\phi \circ m_\Hopf)  	 = 	\exp_{\lcAt}(t (\phi \td 1_\lcA + 1_\lcA \td \phi))			.
   \]
   Note that the exponential function $\exp_{\lcAt}$ is taken in $\lcAt\coloneq\Hom_\K(\Hopf\otimes\Hopf,\lcB)$. 

   This shows that the $1$-parameter subgroups $\gamma_{\phi\circ m_\Hopf}$ and $\gamma_{\phi\td1_\lcA + 1_\lcA\td\phi}$ agree and by Proposition \ref{prop: group_of_units_pro_lie_group} applied to $\lcAt$, we obtain that
   \[
    \phi\circ m_\Hopf = \phi\td1_\lcA + 1_\lcA\td\phi
   \]
   which is equivalent to $\phi$ being an infinitesimal character. This finishes the proof.
  \end{proof}
  \begin{rem} 
   \begin{enumerate}
    \item It is remarkable that Theorem \ref{thm: character_group_is_a_pro_lie_group} holds without any assumption on the given abstract Hopf algebra (in particular, we do not assume that it is graded or connected.)
    \item For a weakly complete commutative algebra $\lcB$ (e.g.\ $\lcB = \K [[X]]$ or $\lcB$ finite dimensional) and a graded and connected Hopf algebra $\Hopf$ the results of Theorem \ref{thm: Char:Lie} and Theorem \ref{thm: character_group_is_a_pro_lie_group} apply both to $\Char{\Hopf}{\lcB}$. 
    
    In this case, the infinite-dimensional Lie group $\Char{\Hopf}{\lcB}$ inherits additional structural properties as a projective limit of finite-dimensional Lie groups. 
    In particular, the regularity of $\Char{\Hopf}{\lcB}$ (cf.\ Theorem \ref{thm: Char_regular}) then follows from \cite{MR2475971}.
    
    For example, these observations apply to the Connes--Kreimer Hopf algebra $\Hopf_{CK}^\K$ and the Butcher group $\Char{\Hopf_{CK}^\K}{\K}$ (see Example \ref{ex: Butcher}).
    In fact, the structure as a pro-Lie group (implicitely) enabled some of the computations made in \cite{BS14} to treat the Lie theoretic properties of the Butcher group.  
   \end{enumerate}\label{rem: lotsastuff}
  \end{rem}
 
%
%

 \begin{appendix}
\section{Locally convex differential calculus and manifolds}\label{app: mfd}

 Basic references for differential calculus in locally convex spaces are \cite{MR1911979,keller1974}.
 
\begin{defn}\label{defn: deriv} 
 Let $\K \in \set{\R,\C}$, $r \in \N \cup \set{\infty}$ and $E$, $F$ locally convex $\K$-vector spaces and $U \subseteq E$ open. 
 Moreover we let $f \colon U \rightarrow F$ be a map.
 If it exists, we define for $(x,h) \in U \times E$ the directional derivative 
 $$df(x,h) \coloneq D_h f(x) \coloneq \lim_{t\rightarrow 0} t^{-1} (f(x+th) -f(x)) \quad (\text{where } t \in \K^\times)$$ 
 We say that $f$ is $C^r_\K$ if the iterated directional derivatives
    \begin{displaymath}
     d^{(k)}f (x,y_1,\ldots , y_k) \coloneq (D_{y_k} D_{y_{k-1}} \cdots D_{y_1} f) (x)
    \end{displaymath}
 exist for all $k \in \N_0$ such that $k \leq r$, $x \in U$ and $y_1,\ldots , y_k \in E$ and define continuous maps $d^{(k)} f \colon U \times E^k \rightarrow F$. 
 If it is clear which $\K$ is meant, we simply write $C^r$ for $C^r_\K$.
 If $f$ is $C^\infty_\K$ we say that $f$ is \emph{smooth}.\footnote{A map $f$ is of class $C^\infty_\C$ if and only if it is \emph{complex analytic} i.e.,
  if $f$ is continuous and locally given by a series of continuous homogeneous polynomials (cf.\ \cite[Proposition 1.1.16]{dahmen2011}).
  We then also say that $f$ is of class $C^\omega_\C$.}    
\end{defn}

\begin{defn}
 Let $E$, $F$ be real locally convex spaces and $f \colon U \rightarrow F$ defined on an open subset $U$. 
 We call $f$ \emph{real analytic} (or $C^\omega_\R$) if $f$ extends to a $C^\infty_\C$-map $\tilde{f}\colon \tilde{U} \rightarrow F_\C$ on an open neighbourhood $\tilde{U}$ of $U$ in the complexification $E_\C$.
\end{defn}

For $\K \in \set{\R,\C}$ and $r \in \N_0 \cup \set{\infty, \omega}$ the composition of $C^r_\K$-maps (if possible) is again a $C^r_\K$-map (cf. \cite[Propositions 2.7 and 2.9]{MR1911979}). 

\begin{defn} Fix a Hausdorff topological space $M$ and a locally convex space $E$ over $\K \in \set{\R,\C}$. 
An ($E$-)manifold chart $(U_\kappa, \kappa)$ on $M$ is an open set $U_\kappa \subseteq M$ together with a homeomorphism $\kappa \colon U_\kappa \rightarrow V_\kappa \subseteq E$ onto an open subset of $E$. 
Two such charts are called $C^r$-compatible for $r \in \N_0 \cup \set{\infty,\omega}$ if the change of charts map $\nu^{-1} \circ \kappa \colon \kappa (U_\kappa \cap U_\nu) \rightarrow \nu (U_\kappa \cap U_\nu)$ is a $C^r$-diffeomorphism. 
A $C_\K^r$-atlas of $M$ is a family of pairwise $C^r$-compatible manifold charts, whose domains cover $M$. Two such $C^r$-atlases are equivalent if their union is again a $C^r$-atlas. 

A \emph{locally convex $C^r$-manifold} $M$ modelled on $E$ is a Hausdorff space $M$ with an equivalence class of $C^r$-atlases of ($E$-)manifold charts.
\end{defn}

 Direct products of locally convex manifolds, tangent spaces and tangent bundles as well as $C^r$-maps of manifolds may be defined as in the finite-dimensional setting. 

\begin{defn}
A $\K$-analytic \emph{Lie group} is a group $G$ equipped with a $C^\omega_\K$-manifold structure modelled on a locally convex space, such that the group operations are $\K$-analytic.
For a Lie group $G$ we denote by $\Lf (G)$ the associated Lie algebra.
\end{defn}

\begin{defn}[Baker--Campbell--Hausdorff (BCH-)Lie groups and Lie algebras] \mbox{}
\begin{enumerate}
 \item A Lie algebra $\mathfrak{g}$ is called \emph{Baker--Campbell--Hausdorff--Lie algebra} (BCH--Lie algebra) if there exists an open $0$-neighbourhood $U \subseteq \mathfrak{g}$ such that for $x, y \in U$ the BCH-series $\sum_{n=1}^\infty H_n (x,y)$ converges and defines an analytic function $U \times U \rightarrow \mathfrak{g}$. 
 (The $H_n$ are defined as $H_1 (x,y) = x +y$, $H_2 (x,y) = \frac{1}{2}\LB[x,y]$ and for $n\geq 3$ by sums of iterated brackets, see \cite[Definition IV.1.5.]{MR2261066}.)
  \item A locally convex Lie group $G$ is called \emph{BCH--Lie group} if it satisfies one of the following equivalent conditions (cf.\  \cite[Theorem IV.1.8]{MR2261066}) 
    \begin{enumerate}[(i)]
     \item $G$ is a $\K$-analytic Lie group whose Lie group with an exponential function which is a local analytic diffeomorphism in $0$.
     \item The exponential map of $G$ is a local diffeomorphism in $0$ and $\Lf (G)$ is a BCH--Lie algebra. 
    \end{enumerate}
\end{enumerate}
\end{defn}

\begin{setup}[{Unit groups of CIAs are BCH--Lie groups \cite[Theorem 5.6]{MR1948922}}] \ \\								\label{thm: glockner_CIA_BCH}
 \emph{Let $\lcA$ be a Mackey complete CIA. Then the group of units $\lcA^\times$ is a $\CoAlg^\omega_\K$-Lie group with the manifold structure endowed from the locally convex space $\lcA$. }

 \emph{The Lie algebra of the group $\lcA^\times$ is $(\lcA,[\cdot,\cdot])$, where $[\cdot,\cdot]$ is the commutator bracket.}
  
 \emph{Moreover, the group $\lcA^\times$ is a Baker--Campbell--Hausdorff--Lie group, i.e.~the exponential map is a local $C^\omega_\K$-diffeomorphism around $0$. 
 This exponential map is given by the exponential series and its inverse is locally given by the logarithm series.}
\end{setup}

To establish regularity of unit groups of CIAs in \cite{MR2997582} a sufficient criterion for regularity was introduced. 
This property called \textquotedblleft ($*$)\textquotedblright \, in loc.cit. is recalled now:
 
 \begin{defn}[(GN)-property]								\label{defn: GN_property}
  A locally convex algebra $\lcA$ is said to satisfy the \emph{(GN)-property}, if for every continuous seminorm $p$ on $\lcA$, there exists a continuous seminorm $q$ and a number $M\geq0$ such that for all $n\in\N$, we have the estimate:
  \[
   \norm{\mu_\lcA^{(n)}}_{p,q} \coloneq \sup \{p(\mu_\lcA^{(n)} (x_1, \ldots ,x_n)) \mid q(x_i) \leq 1,\ 1 \leq i \leq n\} \leq M^n.
  \]
  Here, $\mu_\lcA^{(n)}$ is the $n$-linear map 
   $\func{\mu_\lcA^{(n)}}{\underbrace{\lcA\times\cdots\times\lcA}_{n}}{\lcA}{(a_1,\ldots,a_n)}{a_1 \cdots a_n}$.

 \end{defn}

 \begin{setup}[{\hspace{-1pt}\cite{MR2997582}}]								\label{setup: GN_commutative_locally_m_convex}
 \emph{A locally convex algebra which is either a commutative continuous inverse algebra or locally m-convex has the (GN)-property.}
 \end{setup} 

\begin{setup}[{\hspace{-1pt}\cite[Proposition 3.4]{MR2997582}}]							\label{setup: GN_regular}
 \emph{ Let $\lcA$ be a CIA with the (GN)-property.
  \begin{itemize}
   \item [\textup{(a)}] If $\lcA$ is Mackey complete, then the Lie group $\lcA^\times$ is $C^1$-regular.
   \item [\textup{(b)}] If $\lcA$ is sequentially complete, then $\lcA^\times$ is $C^0$-regular.
  \end{itemize}
  In both cases, the associated evolution map is even $\K$-analytic.}
 \end{setup}

\section{Graded algebra and characters} \label{app: alg}

In this section we recall basic tools from abstract algebra. All results and definitions given in this appendix are well known (see for example \cite{MR0174052,MR1381692,MR1321145,MR0252485}. 
However, for the reader's convenience we recall some details of the construction and proofs.
We assume that the reader is familiar with the definition of algebras, coalgebras and Hopf algebras.

\begin{setup}\label{setup: alg:prop}
For the reader's convenience we summarise important examples of topological algebras and some of their properties discussed in this appendix in the following chart. 
Here the arrows indicate that a given property is stronger than another or that an example possesses the property, respectively.

\addcontentsline{toc}{subsection}{Figure 1: Important properties and examples of locally convex algebras}
\tikzstyle{example} = [rectangle, draw, fill=blue!05, 
    text width=2.45cm, text centered, minimum height=4em]
\tikzstyle{block} = [rectangle, draw,  
    text width=2.45cm, text centered, rounded corners, minimum height=4em]
\tikzstyle{line} = [draw, -latex']
\tikzstyle{white} = [rectangle,
    text width=2.45cm, text centered, rounded corners, minimum height=4em]
 \begin{figure}[h]
 \begin{tikzpicture}[node distance = 1.1cm and .75cm, auto]
    \node [example] (init) {$\K \in \{\R, \C\}$};
    \node [block, below= of init] (findim) {finite dimensional algebra};
    \node [white, right=of findim]  (dummy1) {};

    \node [example, right= of dummy1] (formPot) {$\K [[X]]$\\ (Example \ref{ex: formal power series})};
    \node [block, below= of findim] (Banach) {Banach algebra};    
    \node [white, right=of Banach]  (dummy2) {};
    
    \node [white, right=of formPot]  (dummy3) {};

    \node [block, right= of formPot] (com) {commutative CIA};

    \node [block, below= of Banach] (FSpace) {\Frechet algebra};
    \node [white, below= of FSpace] (dummyx) {};
    \node [block, below= of formPot] (CIA) {Continuous inverse algebra (CIA)};
    \node [block, right= of FSpace] (wc) {weakly complete algebra};
    
    \node [block, right= of dummyx] (complete) {complete};
    \node [block, right= of complete] (sqcomp) {sequentially complete};
    \node [block, right= of sqcomp] (mackey) {Mackey complete};
    \node [white, below= of com] (dummy2){};
    \node [block, below= of CIA] (lmcvx)  {locally \\ $m$-convex};
    \node [block, right= of lmcvx] (GN) {(GN)-property\\ (Definition \ref{defn: GN_property})};
    \path [line] (init) -- (findim);
    \path [line] (init) -| (com);
    \path [line] (findim) -- (Banach);
    \path [line] (wc) -- (lmcvx);
    \path [line] (wc) -- (complete);
    \path [line] (lmcvx) -- (GN);
    \path [line] (com) -- (GN);
    \path [line] (com) -- (CIA);
    \path [line] (Banach) -- (lmcvx);
    \path [line] (formPot) -- (com);
    \path [line] (Banach) -- (FSpace);
    \path [line] (FSpace) -- (complete);
    \path [line] (complete) -- (sqcomp);
    \path [line] (sqcomp) -- (mackey);
    \path [line] (formPot) -- (FSpace);
    \path [line] (Banach) -- (CIA);
     \path [line] (findim) -- ++(3.2cm,-0.0cm)-- ++(-0.0cm,-4.3cm);
     \path [line] (formPot) -- ++(-3.2cm,-0.0cm)-- ++(-0.0cm,-4.3cm);
\end{tikzpicture}

\caption{Important properties and examples of locally convex algebras.}
\end{figure}
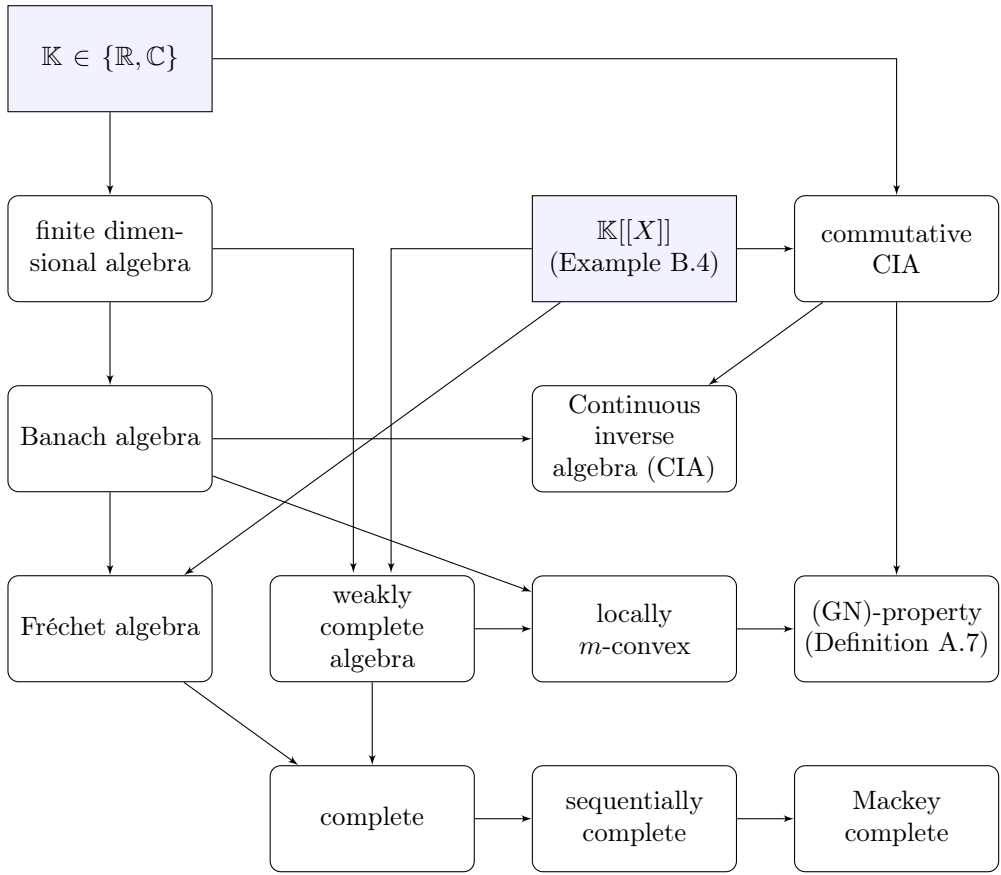

\begin{setup}[Abstract gradings]												
 \ \begin{itemize}
  \item [(a)] Let $\VV$ be an abstract $\K$-vector space. A family of vector subspaces $\seqnN{\VV_n}$ is called (abstract) \emph{$\N_0$-grading} (or just \emph{grading}) of $\VV$, if the canonical linear addition map
		\[
		  \func{\Sigma}{\bigoplus_{n\in\N_0} \VV_n}{E}{\seqnN{v_n}}{\sum_{n\in\N_0}v_n}
		\]
		is an isomorphism of $\K$-vector spaces, i.e.~is bijective.
   \item [(b)] By a \emph{graded algebra}, we mean an abstract $\K$-algebra $\Alg$, together with an abstract grading $\seqnN{\Alg_n}$ of the underlying vector space $\Alg$ such that
		\[
		1_\Alg\in \Alg_0  \hbox{ and } \Alg_n \cdot \Alg_m \subseteq \Alg_{n+m}  \text{ for all } n,m \in \N_0.
		\]
   This implies in particular that $\Alg_0$ is a unital subalgebra and that the projection $\pi_0 \colon \Alg \rightarrow \Alg_0$ onto $\Alg_0$ is an algebra homomorphism.
   \item [(c)] A \emph{graded coalgebra} is an abstract coalgebra $(\CoAlg,\Delta_\CoAlg,\epsilon_\CoAlg)$, together with an abstract grading $\seqnN{\CoAlg_n}$ of the underlying vector space $\CoAlg$ such that for all $n \in \N_0$
		\[
		 \Delta_\CoAlg(\CoAlg_n) \subseteq \bigoplus_{i+j=n} \CoAlg_i \otimes \CoAlg_j \hbox{ and } \bigoplus_{n\geq1} \CoAlg_n \subseteq \ker(\epsilon_\CoAlg).
		\]
	       A graded coalgebra is called \emph{connected} if $\CoAlg_0$ is one dimensional.
   \item [(d)] A \emph{graded Hopf algebra} is an abstract Hopf algebra $(\Hopf,m_\Hopf,u_\Hopf,\Delta_\Hopf,\epsilon_\Hopf)$, together with an abstract grading $\seqnN{\Hopf_n}$ of the underlying vector space $\Hopf$ which is an algebra grading and a coalgebra grading at the same time. 
	       We call a graded Hopf algebra \emph{connected} if $\Hopf_0$ is one dimensional.
	       An element $a \in \Alg_n$ (with $n \in \N_0$) is called \emph{homogeneous (of degree $n$)}. 
 \end{itemize}\label{setup: abstract grading}
\end{setup}

\begin{setup}[Dense Gradings]					\label{setup: dense grading}
  \ \begin{itemize}
   \item [(a)] 
		Let $E$ be a locally convex space. 
		A family of vector subspaces $\seqnN{E_n}$ is called a \emph{dense grading} of $E$, if the canonical linear summation map
		\[
		  \func{\Sigma}{\bigoplus_{n\in\N_0} E_n}{E}{\seqnN{x_n}}{\sum_{n\in\N_0}x_n}
		\]
		can be extended to an isomorphism of topological vector spaces
		\[
		  \smfunc{\overline{\Sigma}}{\prod_{n\in\N_0} E_n}{E}.
		\]
				
		Note that the extension $\overline{\Sigma}$ is unique since $\bigoplus_n E_n$ is dense in $\prod_n E_n$. 
		Furthermore, each $E_n$ is automatically closed in $E$ and $E$ is a \Frechet space if and only if each $E_n$ is a \Frechet space.
   \item [(b)]  By a \emph{densely graded locally convex algebra}, we mean a locally convex algebra $\lcA$, 
		  together with a dense grading of the underlying locally convex space $\lcA$ such that
		\[
		 \lcA_n \cdot \lcA_m \subseteq \lcA_{n+m}\text{ for all } n,m \in \N_0 \text{ and } 1_{\lcA}\in A_0.
		\]
		This implies in particular that $\lcA_0$ is a closed unital subalgebra and that the projection $\smfunc{\pi_0}{\lcA}{\lcA_0}$ is a continuous algebra homomorphism.
		
                Denote the kernel of $\pi_0$ by $\IdealA\coloneq\ker(\pi_0)=\overline{\bigoplus_{n\geq 1} \lcA_n}$. 
                The kernel $\IdealA$ is a closed ideal.
		Each element in $\lcA$ has a unique decomposition $a=a_0 + b$ with $a_0\in\lcA_0$ and $b\in\IdealA$.
  \end{itemize}
\end{setup}

Note that \emph{densely graded algebras} (see \ref{setup: dense grading}(b)) are not included in Diagram \ref{setup: alg:prop} as every locally convex algebra $\lcA$ admits the trivial grading $\lcA_0=\lcA$ and $\lcA_n=0$ for $n\geq1$.

\begin{ex}[Formal power series]												\label{ex: formal power series}
 Let $\K[[X]]$ be the algebra of formal power series in one variable. 
 We give this algebra the topology of pointwise convergence of coefficients, i.e.~the initial topology with respect to the coordinate maps:
 \[
  \func{\kappa_n}{\K[[X]]}{\K}{\sum_{k=0}^\infty c_k X^k}{c_n.}
 \]
 As a topological vector space, the algebra $\K[[X]]$ is isomorphic to the \Frechet space $\K^{\N_0}=\prod_{k=0}^\infty \K$
 We see that $\K[[X]]$ is a densely graded algebra with respect to the grading
 $
  \seqnN{\K X^n=\setm{c X^n}{c\in\K}}.
 $
 Note that $\K[[X]]$ is a CIA by Lemma \ref{lem: unit_groups_of_graded_algebras} (b).
\end{ex}

\end{setup}
In the following, we will identify a densely graded locally convex algebra $\lcA$ with the product space $\prod_{n=0}^\infty\lcA_n$ such that each element $a\in\lcA$ is a tuple $\seqnN{a_n}$.

\begin{lem}[Functional calculus for densely graded algebras]					\label{lem: functional_caluculus}
  Let $\lcA=\prod_{n=0}^\infty \lcA_n$ be a densely graded locally convex algebra with $\IdealA = \ker\pi_0$. 
  Then there exists a unique continuous map
  \[
   \nnfunc{\K[[X]]\times \IdealA }{\lcA}{(f,a)}{f[a]}
  \]
  such that for all $a\in\IdealA$, we have $X[a]=a$ and the map
  \[
   \nnfunc{\K[[X]]}{\lcA}{f}{f[a]}
  \]
  is a morphism of unital algebras.
  If $f=\sum_{k=0}^\infty c_k X^k$ and $a=\seqnN{a_n}$ with $a_0=0$ are given, the following explicit formula holds:
  \begin{equation}\label{eq: fun:explicit}
   f[a] = \seqnN{	\sum_{k=0}^n c_k \sum_{\substack{\alpha\in\N^k \\ 
                                                         \abs{\alpha}=n}} a_{\alpha_1} \cdot \ldots \cdot a_{\alpha_k}	}.
  \end{equation}
  Furthermore, the map $\nnfunc{\K[[X]]\times \IdealA }{\lcA}{(f,a)}{f[a]}$ is a $C^\omega_\K$-map (cf.\ Appendix \ref{app: mfd}).
\end{lem}

\begin{proof}
 First of all, the explicit formula is well-defined and continuous on $\K[[X]]\times \IdealA$ since every component is a continuous polynomial in finitely many evaluations of the spaces $\K[[X]]$ and $\IdealA$.
 As $\lcA$ is densely graded and thus isomorphic to the locally convex product of the spaces $\lcA_n, n\in \N_0$, this implies that the map is continuous. 
 In fact, this already implies that the map is $C^\omega_\K$ for $\K=\C$. 
 For $\K=\R$ one has to be a little bit more careful since there exist maps into products which are not $C^\omega_\R$ although every component is $C^\omega_\R$ (c.f.\ \cite[Example 3.1]{MR2402519}). 
 However, if each component is a continuous polynomial, the real case follows from the complex case as real polynomials complexify to complex polynomials by \cite[Theorem 3]{MR0313810}.
 
 Let $a\in\IdealA$ be a fixed element. It remains to show that $\nnfunc{\K[[X]]}{\lcA}{f}{f[a]}$ is an algebra homomorphism.
 By construction it is clear that $f\mapsto f[a]$ is linear and maps $X^0$ to $1_\lcA$ and $X^1$ to $a$.
 Since $f$ is continuous and linear, it suffices to establish the multiplicativity for series of the form $X^N$, i.e.~it suffices to prove that
 \[
  (X^N[a] )\cdot (X^M[a]) = X^{N+M}[a]
 \]
 which follows from the easily verified fact that $X^N[a]=a^N$.

 To establish uniqueness of the map obtained, we remark the following: 
 A continuous map on $\K[[X]]$ is determined by its values on the dense space of polynomials $\K[X]$, and an algebra homomorphism on $\K[X]$ is determined by its value on the generator $X$. 
 In the case at hand this value has to be $a$.
\end{proof}

\begin{lem}[Exponential and logarithm]									\label{lem: exp_and_log}
 Consider the formal power series
 \[
  \exp(X)\coloneq\sum_{k=0}^\infty \frac{X^k}{k!} \quad \hbox{ and } \quad \log(1+X) \coloneq \sum_{k=1}^\infty (-1)^{k+1} \frac{X^k}{k}.
 \]
 Let $\lcA$ be a densely graded locally convex algebra.
 The exponential function restricted to the closed vector subspace $\IdealA$
 \[
  \exp_\lcA \colon \nnfunc{\IdealA}{1_\lcA + \IdealA}{a}{\exp[a]}
 \]
 is a $C^\omega_\K$-diffeomorphism with inverse
 \[
  \log_\lcA \colon \nnfunc{1_\lcA+\IdealA}{ \IdealA}{(1_\lcA+a)}{\log(1+X)[a].}
 \]
\end{lem}
\begin{proof}
 The maps are $C^\omega_\K$ by Lemma \ref{lem: functional_caluculus}. 
 As formal power series $E(X) \coloneq \exp (X)-1$ and $L(X)\coloneq \log (X+1)$ are inverses with respect to composition of power series. 
 Note that $\K [[X]]$ is a densely graded locally convex algebra with $E(X) , L(X) \in \cI_{\K [[X]]}$. 
 Hence, we can apply the functional calculus of Lemma \ref{lem: functional_caluculus} to $\K [[X]]$ and the elements $E(X) , L(X)$.
 This yields for $a \in \IdealA$ the identity 
 $$E[L [a]] = (E\circ L)[a] = X[a] = a.$$
 Note that similarly one proves $L [E[a]] = a$. 
 Therefore, $\exp_\lcA$ and $\log_\lcA$ are mutually inverse mappings. 
\end{proof}

\begin{lem} \label{lem: exp:com}
 Let $\lcA = \prod_{n \in \N_0} \lcA_n$ be a densely graded locally convex algebra with exponential map $\exp_\lcA$. 
 Then the following assertions hold: 
  \begin{enumerate} 
   \item[\textup{(a)}] For $a,b \in \IdealA$ with $ab = ba$ we have $\exp_\lcA (a+b) = \exp_\lcA (a)\exp_\lcA(b)$.
   \item[\textup{(b)}] The derivative of $\exp_\lcA$ at $0$ is $T_0 \exp_\lcA = \id_{\IdealA}$.
  \end{enumerate}
    
\end{lem}

\begin{proof}
 By construction of $\exp_\lcA$ we derive from Lemma \ref{lem: functional_caluculus} for $x \in \IdealA$ the formula $\exp_\lcA (x) = \lim_{N \rightarrow \infty} \sum_{k = 0}^N \frac{x^k}{k!}$.
 The algebra $\lcA$ is densely graded and we have for every $n \in \N_0$ a continuous linear projection $\pi_n \colon \lcA \rightarrow \lcA_n$.
 By definition we have for $x \in \IdealA$ that $\pi_0 (x) = 0$. 
 Hence, the definition of a densely graded algebra implies for $x,y \in \IdealA$ that $\pi_j (x^ky^l) = 0$ if $k+l > j$.
 \begin{enumerate}
  \item As $a$ and $b$ commute we can compute as follows: 
  \begin{equation} \label{eq: comm}\begin{aligned}
   &\exp_\lcA (a)\exp_\lcA (b) = \lim_{N \rightarrow \infty} \left(\sum_{k+l \leq N} \frac{a^kb^l}{k!l!} + \vphantom{\sum_{\substack{l+k > N , \\ l,k \leq N}}} \right. \underbrace{\sum_{\substack{l+k > N , \\ l,k \leq N}} \frac{a^kb^l}{k!l!}}_{\equalscolon S_N}\left.\vphantom{\sum_{\substack{l+k > N , \\ l,k \leq N}}}\right) \\
			      =& \lim_{N \rightarrow \infty} \left(\sum_{n=0}^N \sum_{k+l = n} \frac{a^k b^l}{k!l!} + S_N\right) = \lim_{N \rightarrow \infty} \left(\sum_{n=0}^N \frac{(a + b)^n}{n!} + S_N\right)
			      \end{aligned}
  \end{equation}
 The first summand in the lower row converges to $\exp_\lcA (a+b)$. 
 
 Now the definition of $S_N$ shows that $\pi_j (S_N) =0$ if $N \geq j$.
 Apply the continuous map $\pi_j$ to both sides of \eqref{eq: comm} for $j \in \N_0$ to derive 
  \begin{displaymath}
   \pi_j (\exp_\lcA (a)\exp_\lcA (b)) = \lim_{N \rightarrow \infty} \left( \pi_j \left(\sum_{n=0}^N \frac{(a + b)^n}{n!}\right) + \pi_j (S_N)\right)
  \end{displaymath}
 On the right hand side the second term vanishes if $N > j$. 
 Thus in passing to the limit we obtain $\pi_j (\exp_\lcA (a)\exp_\lcA (b)) = \pi_j (\exp_\lcA (a + b))$ for all $j \in \N_0$.
 \item The image of $\exp_\lcA$ is the affine subspace $1_\lcA + \IdealA$ whose tangent space (as a submanifold of $\lcA$) is $\IdealA$. 
 We can thus identify $\exp_\lcA$ with $F \colon \IdealA \rightarrow \lcA , a \mapsto \exp[a]$ to compute $T_0 \exp_\lcA$ as $d F(0; \cdot)$. 
 The projections $\pi_n, n \in \N_0$ are continuous linear, whence it suffices to compute $d \pi_n \circ F (0;\cdot) = \pi_n dF (0;\cdot)$ for all $n \in \N_0$.
 Now for $n \in \N_0$ and $a \in \IdealA$ compute the derivative in $0$:
  \begin{align*}
   \pi_n dF (0;a) &\stackrel{\hphantom{\eqref{eq: fun:explicit}}}{=} d \pi_n \circ F (0;a) = \lim_{t \rightarrow 0} t^{-1} (\pi_n \circ F (ta) - \pi_n \circ F(0)) \\
		  &\stackrel{\eqref{eq: fun:explicit}}{=}  \lim_{t \rightarrow 0} \sum_{k = 1}^n \frac{1}{k!} \sum_{\overset{\alpha \in \N^k}{|\alpha| = n}} t^{k-1} a_{\alpha_1} \cdot \ldots \cdot a_{\alpha_k} = a_n = \pi_n \circ \id_{\IdealA} (a) \qedhere
  \end{align*}
 \end{enumerate}
\end{proof}

\begin{lem}[Unit groups of densely graded algebras]									\label{lem: unit_groups_of_graded_algebras}
 Let $\lcA=\prod_{n=0}^\infty\lcA_n$ be a densely graded locally convex algebra.
 \begin{itemize}
  \item [\textup{(a)}] An element $a\in\lcA$ with decomposition $a = a_0 + b$ is invertible in $\lcA$ if and only $a_0$ is invertible in $\lcA_0$.
  \item [\textup{(b)}] The algebra $\lcA$ is a CIA if and only if $\lcA_0$ is a CIA.
  In particular, if $\lcA_0=\K$, then $\lcA$ is a CIA.
 \end{itemize}
\end{lem}
\begin{proof} 
  \begin{itemize}
   \item[(a)] The map $\smfunc{\pi_0}{\lcA}{\lcA_0}$ is an algebra homomorphism. 
    This implies that invertible elements $a\in\lcA$ are mapped to invertible elements $a_0\in\lcA_0$.
    For the converse, take an element $a\in\lcA$ with decomposition $a=a_0  + b$ with $b\in\IdealA$ and $a_0$ is invertible in $\lcA_0$.
    Then we may multiply by $a_0^{-1}$ from the left and obtain the equality
    \[
      a_0^{-1} a = 1 + a_0^{-1} b.
    \]
    This shows that $a$ is invertible if we are able to prove that $1+a_0^{-1}b$ is invertible. 
    Apply the formal power series
    \[
      (1-X)^{-1}=\sum_{k=0}^\infty X^k
    \]
    to the element $-a_0^{-1}b\in\IdealA$ and obtain the inverse of $1+ a_0^{-1}b$.
    \item[(b)]  We have seen in part (a) that the units in the algebra $\lcA$ satisfy
    \[
      \lcA^\times = \pi_0^{-1}(\lcA_0^\times)
    \]
    and hence one of the unit groups is open if and only if the other one is open. 
    It remains to establish that continuity of inversion in $\lcA_0^\times$ implies continuity of inversion in $\lcA^\times$.
    In part (a) we have seen that inversion of $a = a_0 + b$ in $\lcA$ is given by
    \[
      a^{-1} = \left((1-X)^{-1}\right) \left[ -a_0^{-1} b \right] \cdot a_0^{-1}
    \]
    So, the continuity of inversion in $\lcA$ follows from the continuity of inversion in $\lcA_0$ and the continuity of the functional calculus (Lemma \ref{lem: functional_caluculus}). \qedhere
  \end{itemize}
\end{proof}

\begin{lem}\label{lemma: red2CIA}
 Let $\lcA = \prod_{n \in \N_0} \lcA_n$ be a densely graded locally convex algebra over the field $\K \in \{ \R , \C \}$. 
 Then $\Cut{\lcA} \coloneq \K 1_{\lcA_0} \times \prod_{n \in \N} \lcA_n \subseteq \lcA$ is a closed subalgebra of $\lcA$. 
 Furthermore, $\Cut{\lcA}$ is densely graded with respect to the grading induced by $(\lcA_n)_{n\in \N_0}$ and $\Cut{\lcA}$ is a CIA.
\end{lem}

\begin{proof}
 Clearly $\Cut{\lcA}$ is a subalgebra of $\lcA$ and the subspace topology turns this subalgebra into a locally convex algebra over $\K$. 
 By definition $\Cut{\lcA}$ is the product of the (closed) subspaces $(\Cut{\lcA}_n)_{n \in \N_0}$. Hence $\Cut{\lcA}$ is a closed subalgebra of $\lcA$ with dense grading $(\Cut{\lcA}_n)_{n \in \N_0}$.
 Finally we have the isomorphism of locally convex algebras $\Cut{\lcA}_0 = \K 1_{\lcA_0} \cong \K$. Hence $\Cut{\lcA}_0$ is a CIA and thus $\Cut{\lcA}$ is a CIA by Lemma \ref{lem: unit_groups_of_graded_algebras} (b).
\end{proof}

\subsection*{Auxiliary results concerning characters of Hopf algebras}
\addcontentsline{toc}{subsection}{Auxiliary results concerning characters of Hopf algebras}

  Fix for the rest of this section a $\K$-Hopf algebra $\Hopf=(\Hopf,m_\Hopf,u_\Hopf,\Delta_\Hopf,\epsilon_\Hopf,S_\Hopf)$ and a commutative locally convex algebra $\lcB$.
  Furthermore, we assume that the Hopf algebra $\Hopf$ is graded and connected, i.e.\ $\Hopf = \bigoplus_{n \in \N_0} \Hopf$ and $\Hopf_0 \cong \K$.
  The aim of this section is to prove that the exponential map $\exp_\lcA$ of $\lcA \coloneq \Hom_\K (\Hopf, \lcB)$ restricts to a bijection from the infinitesimal characters to the characters.

 \begin{lem}[Cocomposition with Hopf multiplication]											\label{lem: dual_comult}
  Let $\Hopf \otimes \Hopf$ be the tensor Hopf algebra (cf.\ \cite[p. 8]{MR1381692}).
  With respect to the topology of pointwise convergence and the convolution product, the algebras
   \begin{align*}
    \lcA \coloneq \Hom_\K (\Hopf,\lcB) \quad \quad \lcAt \coloneq \Hom_\K (\Hopf \otimes \Hopf , \lcB) 
   \end{align*}
   become locally convex algebras (see \ref{setup: maps on coalgebra}). 
  This structure turns
  \[
   \func{\cdot \circ m_\Hopf}{\Hom_\K(\Hopf,\lcB)}{\Hom_\K(\Hopf\otimes \Hopf,\lcB)}{\phi}{\phi\circ m_\Hopf.}
  \]
  into a continuous algebra homomorphism.
 \end{lem}

 \begin{proof}
  From the usual identities for the structure maps of Hopf algebras (cf.\ \cite[p.7 Fig. 1.3]{MR1381692}) it is easy to see that $\cdot \circ m_\Hopf$ is an algebra homomorphism.
  Clearly $\cdot \circ m_\Hopf$ is continuous with respect to the topologies of pointwise convergence.
 \end{proof}

  \begin{lem}\label{lem: exp:bij}
  The analytic diffeomorphism $\smfunc{\exp_{\lcA}}{\IdealA}{1+\IdealA}$ maps the set of infinitesimal characters $\InfChar{\Hopf}{\lcB}$ bijectively onto the set of characters $\Char{\Hopf}{\lcB}$.\footnote{It is hard to find a complete proof in the literature, whence we chose to include a proof for the reader's convenience.}
 \end{lem}
 
 \begin{proof}
  \newcommand{\td}{\diamond}
  We regard the tensor product $\Hopf\otimes \Hopf$ as a graded and connected Hopf algebra with the tensor grading, i.e.\ $\Hopf \otimes \Hopf = \bigoplus_{n \in \N_0} (\Hopf \otimes \Hopf)_n$ where for all $n \in \N_0$ the $n$th degree is defined as $(\Hopf \otimes \Hopf)_n = \bigoplus_{i+j = n} \Hopf_i \otimes \Hopf_j$.
  
  The set of linear maps $\lcAt\coloneq \Hom_{\K}(H\otimes H,\lcB)$ is a densely graded locally convex algebra with the convolution product $\star_{A_\otimes}$. 
  Let $\func{m_\lcB}{\lcB\otimes \lcB}{\lcB}{b_1\otimes b_2}{b_1\cdot b_2}$ be the multiplication in $\lcB$. 
  We define a continuous bilinear map
  \[
   \beta \colon \nnfunc{\lcA \times \lcA}{\lcAt}{(\phi,\psi)}{\phi \td \psi \coloneq  m_\lcB \circ (\phi \otimes \psi)}
  \]
  and will now check that $\beta$ is continuous.
  Let $c=\sum_{k=1}^n c_{1,k}\otimes c_{2,k}\in \Hopf\otimes\Hopf$ be a fixed element. It remains to show that $\phi\td\psi (c)$ depends continuously on $\phi$ and $\psi$:
  \begin{align*}
   (\phi \td \psi) (c)	  &  	= 		m\circ (\phi\otimes\psi)\left(c\right)
			  	= 		m\circ ( \phi\otimes\psi) \left(	\sum_{k=1}^n c_{1,k}\otimes c_{2,k}	\right) 
			\\&	= \sum_{k=1}^n 	m\left(\phi(c_{1,k}) \otimes \psi(c_{2,k}) \right)
				= \sum_{k=1}^n 	\phi(c_{1,k}) \cdot \psi(c_{2,k})
  \end{align*}
  This expression is continuous in $(\phi,\psi)$ since point evaluations are continuous as well as multiplication in the locally convex algebra $\lcB$. 

  We may use this operation to write the convolution in $A$ as $\star_A = \beta \circ \Delta$ and obtain
  \begin{equation}\label{eq: multifalt}
   (\phi_1 \td \psi_1 )\star_{A_\otimes} (\phi_2 \td \psi_2) = (\phi_1\star_A \phi_2) \td (\psi_1 \star_A \psi_2).
  \end{equation}
  Recall, that $1_\lcA\coloneq u_\lcB\circ\epsilon_\Hopf$ is the neutral element of the algebra $\lcA$.
  From equation \eqref{eq: multifalt}, it follows at once, that the continuous linear maps
  \begin{equation}\label{eq:alg:hom}
   \beta (\cdot, 1_A) \colon \nnfunc{\lcA}{\lcAt}{\phi}{\phi \td 1_\lcA} \quad \text{ and } \quad \beta (1_A, \cdot) \colon \nnfunc{\lcA}{\lcAt}{\phi}{1_\lcA \td \phi} 
  \end{equation}
  are continuous algebra homomorphisms. 
  
  We will now exploit $\td$ to prove that the bijection $\smfunc{\exp_\lcA}{\IdealA}{1_\lcA + \IdealA}$ (see Lemma \ref{lem: exp_and_log}) maps the set $\InfChar{H}{\lcB}$ onto $\Char{H}{\lcB}$. 
  Let $\phi\in \IdealA$ be given and recall: 
    \begin{enumerate}
     \item The Hopf algebra product $m_\Hopf$ maps $\Hopf_0 \otimes \Hopf_0$ into $\Hopf_0$. Now $\Hopf_0 \otimes \Hopf_0 = (\Hopf \otimes \Hopf)_0$ (tensor grading) entails for $\phi \in \cI_\lcA$ that $\phi \circ m_\Hopf \in \cI_{\lcA_\otimes}$.
     \item From \eqref{eq: multifalt} we derive that $(\phi \td 1_\lcA) \star_{\lcA_\otimes} (1_\lcA \td \phi) = \phi \td \phi =  (1_\lcA \td \phi) \star_{\lcA_\otimes}  (\phi \td 1_\lcA)$. 
     If $\phi \in \lcA$ is an infinitesimal character then $\phi \circ m_\Hopf = \phi \td 1_\lcA + 1_\lcA \td \phi$. 
     Together with (a) this shows that Lemma \ref{lem: exp:com} is applicable and as a consequence 
      \begin{equation}\label{eq: com:ptm}
       \exp_{\lcAt} (\phi \td 1_\lcA + 1_\lcA \td \phi) = \exp_{\lcAt} (\phi \td 1_\lcA) \star_{\lcA_\otimes} \exp_{\lcAt} (1_\lcA \td \phi).
      \end{equation}

    \end{enumerate}  
  Note that it suffices to check multiplicativity of $\exp_\lcA (\phi)$ as $\exp_{\lcA}(\phi)(1_\Hopf)=1_\lcB$ is automatically satisfied. 
  To prove the assertion we establish the following equivalences:
  \begin{align*}
	\phi 	\in \InfChar{\Hopf}{\lcB}	& \stackrel{\text{Def}}{\iff} \phi \circ m_\Hopf = \phi \td 1_\lcA + 1_\lcA \td \phi \\
						& \stackrel{\text{(a)}}{\iff}     \exp_{\lcAt}(\phi \circ m_\Hopf)  	 = 	\exp_{\lcAt}(\phi \td 1_\lcA + 1_\lcA \td \phi)							\\
						& \stackrel{\eqref{eq: com:ptm}}{\iff} \exp_{\lcAt}(\phi \circ m_\Hopf)  	 = 	\exp_{\lcA_\otimes}(\phi \td 1_\lcA) \star_{\lcA_\otimes} \exp_{\lcA_\otimes}(1_\lcA \td \phi)				\\
						&\stackrel{\eqref{eq:alg:hom}}{\iff} \exp_{\lcAt}(\phi \circ m_\Hopf)  	 = 	\bigl( \exp_{\lcA}(\phi) \td 1_\lcA  \bigr)\star_{\lcA_\otimes} \bigl(1_\lcA \td \exp_{\lcA}(\phi) \bigr)	\\
						&\stackrel{\eqref{eq: multifalt}}{\iff} \exp_{\lcAt}(\phi \circ m_\Hopf)  	 = 	\bigl( \exp_{\lcA}(\phi)\star_{\lcA} 1_\lcA \bigr) \td \bigl(1_\lcA  \star_{\lcA} \exp_{\lcA}(\phi) \bigr)	\\
						&\iff \exp_{\lcAt}(\phi \circ m_\Hopf)  	 = 	\exp_{\lcA}(\phi) \td  \exp_{\lcA}(\phi) 							\\
						&\stackrel{\text{\ref{lem: dual_comult}}}{\iff} 
						      \exp_{\lcA}(\phi) \circ m_\Hopf  		 = 	\exp_{\lcA}(\phi) \td  \exp_{\lcA}(\phi) 	\\
						&\stackrel{\text{Def}}{\iff}
						      \exp_{\lcA}(\phi) \in \Char{\Hopf}{\lcB}\qedhere
  \end{align*}
 \end{proof}

\begin{rem}\label{rem: salv:wc}
 The chain of equivalences in the proof of Lemma \ref{lem: exp:bij} uses the dense grading of $\lcA = \Hom_\K (\Hopf,\lcB)$ twice to show that the first to third lines are equivalent. 
 However, for arbitrary $\Hopf$ and weakly complete $\lcB$ the second and third line are still equivalent: 
 In this case $\lcA_{\otimes}$ is weakly complete and Lemma \ref{lem: fundamental_lemma_of_weakly_complete_algebras} (c) allows us to embed $\lcA_\otimes$ into $P \coloneq \prod_{i \in I} \lcA_i$ (product of Banach algebras in the category of topological algebras).
 This implies that the formula $\exp_{\lcA_\otimes} (a+b) = \exp_{\lcA_\otimes}(a)\exp_{\lcA_\otimes} (b)$ used in \eqref{eq: com:ptm} still holds as the power series defining $\exp_\lcA$ converges on $P$ and satisfies the formula (which is component-wise true in every Banach algebra).    
\end{rem}

\section{Weakly complete vector spaces and duality}						\label{app: weakly complete}

The purpose of this section is to exhibit the duality between the category of abstract vector spaces and the category of weakly complete topological vector spaces. 
Although none of this is needed for the results of this paper, many ideas of this paper appear to be more natural in this wider setting.

Throughout this section, let $\K$ be a fixed Hausdorff topological field. Although, in this paper, we are only interested in the cases $\K=\R$ and $\K=\C$, the statements in this appendix hold for an arbitrary Hausdorff field of any characteristic, including the discrete ones.\footnote{In functional analysis, usually only $\R$ and $\C$ with their usual field topologies are considered, where in algebra usually an arbitrary field with the discrete topology is considered. Our setup includes both cases (and many more, e.g.~the $p$-adic numbers, etc.).}
We start with a definition.

\begin{defn}								\label{def: weakly_complete_space}
 A topological vector space $\EE$ over the topological field $\K$ is called \emph{weakly complete topological vector space} (or \emph{weakly complete space} for short) if one of the following equivalent conditions is satisfied:
 \begin{itemize}
  \item [(a)] There exists a set $I$ such that $\EE$ is topologically isomorphic to $\K^I$.
  \item [(b)] There exists an abstract $\K$-vector space $\VV$ such that $\EE$ is topologically isomorphic to $\VV^*\coloneq\Hom_\K(\VV,\K)$ with the weak*-topology
  \item [(c)] The space $E$ is the projective limit of its finite-dimensional subspaces, where each $n$-dimensional subspace being topologically isomorphic to $\K^n$
 \end{itemize}
 For the case $\K=\R$ or $\K=\C$, these conditions are also equivalent to the following conditions:
 \begin{itemize}
  \item [(d)] The space $\EE$ is locally convex and is complete with respect to the weak topology.
  \item [(e)] The space $\EE$ is locally convex, it carries its weak topology and is complete with this topology.
 \end{itemize}
\end{defn}
The proof that (a)$\iff$(b)$\iff$(c)  can be found in \cite[Appendix 2]{MR2337107}). 
The characterisations (d) and (e) justify the name \emph{weakly complete}.

\begin{rem}
 Part (b) of the preceding definition tells us that the \emph{algebraic dual} $\VV^*$ of an abstract vector space $\VV$ becomes a weakly complete topological vector space with respect to the weak*-topology, i.e.~the topology of pointwise convergence.
 
 Conversely, given a weakly complete vector space $E$, we can consider the \emph{topological dual} $\EE'$ of all \emph{continuous} linear functionals. Although there are many vector space topologies on this topological dual, we will always take $E'$ as an abstract vector space.
\end{rem}

One of the main problems when working in infinite-dimensional linear (and multilinear) algebra is that a vector space $\VV$ is no longer isomorphic to its bidual $(\VV^*)^*$. 
The main purpose of this section is to convince the reader that the reason for this bad behaviour of the bidual is due to the fact that the wrong definition of a bidual is used (at least for infinite-dimensional spaces). 

If we start with an abstract vector space $\VV$, then its dual is a weakly complete space $\VV^*$ and consequently, one should \emph{not} take the algebraic dual $(\VV^*)^*$ but the topological dual $(\VV^*)'$ which is the natural choice.
For a finite dimensional space the construction coincides with the usual definition of the bidual.
In the general case however, the so obtained bidual is now canonically isomorphic to the original space as the following proposition shows:

\begin{prop}[Duality and Reflexivity]											\label{prop: duality_and_reflexivity}
 Let $\EE$ be a weakly complete space and let $\VV$ be an abstract vector space. 
 There are natural isomorphisms
 \[
   \Func{\eta_\EE}{\EE}{  ( \EE' )^* }{x}{(\func{ \eta_\EE(x)\coloneq \phi_x}{\EE'}{\K}{\lambda}{\lambda(x)})}
 \]
 and
 \[
   \Func{\eta_\VV}{\VV}{  ( \VV^*) ' }{v}{(\func{ \eta_\VV(v)\coloneq \lambda_v}{\VV^*}{\K}{\phi}{\phi(v)}).}
 \]
\end{prop}
\begin{proof}[Proof (Sketch)]
 Let $\EE$ be a weakly complete vector space. We may assume that $\EE=\K^I$ for a set $I$. 
 Then each projection map $\smfunc{\pi_i}{\K^I}{\K}$ on the $i$-th component is an element in $\EE'$. It is easy to see that $(\pi_i)_{i\in I}$ is in fact a basis of the abstract vector space $\EE'$. This means that the algebraic dual of $\EE'$ is topologically isomorphic to $\K^I$. Using this identification, one can check that the map $\eta_\EE$ is the identity.

 Similarly, let $\VV$ be an abstract vector space. By Zorn's Lemma, pick a basis $(b_i)_{i\in I}$. Then the dual space $\VV^*$ is isomorphic to $\K^I$. And therefore, the dual of that one $(\VV^*)'$ has a basis $(\pi_i)_{i\in I}$. Under this identification, the linear map $\eta_\VV$ is the identity.
\end{proof}

\begin{setup}[The weakly complete tensor product]
 One way to understand Proposition \ref{prop: duality_and_reflexivity} is that every element $x$ in a weakly complete space $\EE$ can be identified with a linear functional $\phi_x=\eta_\EE(x)\in (\EE')^*$ on the abstract vector space $\EE'$. 
 This enables us to define a tensor product of two elements $x\in\EE$ and $y\in\FF$ as the tensor product of the corresponding linear functionals
 \[
  \Func{x\otimes y \coloneq \phi_x \otimes \phi_y}{\EE'\otimes\FF'}{\K}{\lambda\otimes \mu }{\phi_x (\lambda) \cdot \phi_y(\mu) = \lambda(x) \cdot \mu(y).}
 \]
 This element $x\otimes y$ is now a linear functional on the abstract vector space $\EE'\otimes\FF'$. This motivates the definition:
 \[
  \EE\wcotimes \FF \coloneq (\EE'\otimes\FF')^*.
 \]
 If the spaces $\EE$ and $\FF$ are of the form $\EE=\K^I$ and $\FF=\K^J$, it is easy to verify that the space $\K^I\wcotimes\K^J = (\EE'\otimes\FF')^*$ is canonically isomorphic to $\K^{I\times J}$.
 This could have been taken as the definition of the weakly complete tensor product in the first place.
 However, the definition we chose has the advantage that is independent of the choice of coordinates, i.e.~the specific isomorphisms $\EE\cong\K^I$ and $\FF\cong\K^J$, respectively.
\end{setup}

\begin{prop}[The universal property of the weakly complete tensor product]
 Let $\EE,\FF,\HH$   be weakly complete spaces and let $\smfunc{\beta}{\EE\times \FF}{\HH}$ be a continuous bilinear map.
 Then there exists a unique continuous linear map $\smfunc{\beta^\sim}{\EE\wcotimes\FF}{\HH}$ such that the following diagram commutes:
 \[
   \xymatrix{
		  \EE \times \FF \ar[rr]^{\beta}\ar[d]_\otimes		& & 	\HH		\\
		  \EE\wcotimes \FF \ar[rru]_{\beta^\sim}
   }
 \]
 For the case $\K=\R$ or $\K=\C$, this universal property also holds for arbitrary complete locally convex spaces $\HH$, showing that this weakly complete tensor product is just a special case of the usual \emph{projective tensor product} for locally convex vector spaces.
\end{prop}

\begin{setup}[Duality of the monoidal categories $\WCCat$ and $\VSCat$] \mbox{}
\begin{itemize}
 \item[(i)] Denote the category of weakly complete spaces and continuous linear maps by $\WCCat$.
 Together with the weakly complete tensor product  and the ground field $\K$ as unit object, we obtain a monoidal category
 $(\WCCat,\wcotimes,\K)$.
 \item[(ii)] Denote the monoidal category of abstract vector spaces, abstract linear maps, the usual abstract tensor product and the ground field as unit object by
 $(\VSCat,\otimes,\K)$.
\end{itemize}
 These two categories are dual to each other. 
 The dualities are given by the contravariant monoidal functors \emph{algebraic dual}
 \[
   \BiFunc{(\cdot)^*}{\VSCat}{\WCCat}{\VV}{\VV^*}{(\smfunc{\Phi}{\VV}{\WW})}{(\func{\Phi^*}{\WW^*}{\VV^*}{\phi}{\phi\circ \Phi)} }
 \]
 and \emph{topological dual}
 \[
   \BiFunc{(\cdot)' }{\WCCat}{\VSCat}{\EE}{\EE '}{(\smfunc{ T  }{\EE}{\FF})}{(\func{  T  '}{\FF '}{\EE'}{\lambda}{\lambda\circ T)} }
 \]
 (cf.\ Proposition \ref{prop: duality_and_reflexivity}).
 The duality interchanges direct sums in the abstract category with direct products in the weakly complete category, hence graded vector spaces (\ref{setup: abstract grading}) are assigned to densely graded vector spaces (\ref{setup: dense grading}).
 For more information about this duality, we refer to \cite[page 679]{MR2035107} and to \cite[Appendix 2]{MR2337107}.
\end{setup}
 
\begin{setup}
 We can naturally define weakly complete algebras, weakly complete coalgebras, weakly complete bialgebras and weakly complete Hopf algebras using the weakly complete tensor product in $(\WCCat,\wcotimes,\K)$. By duality, we get the correspondence:\medskip

\setlength{\extrarowheight}{1.5pt}
\begin{tabular}{|p{4.6cm}|p{7.2cm}|} \hline
	{\textbf{Abstract world}} $(\VSCat,\otimes,\K)$	& 	{\textbf{Weakly complete world}}  $(\WCCat,\widetilde{\otimes},\K)$		\\
     \hline
			abstract vector space	& weakly complete vector space			\\
			\quad linear map		& \quad continuous linear map				\\
			\quad graded vector space	& \quad densely graded weakly complete vector space	\\
     \hline
			abstract coalgebra	& weakly complete algebra			\\
     \hline
			abstract algebra	& weakly complete coalgebra			\\
     \hline
			abstract bialgebra	& weakly complete bialgebra			\\
     \hline
			abstract Hopf algebra	& weakly complete Hopf algebra			\\
			\quad characters		& \quad group like elements				\\
			\quad infinitesimal characters	& \quad primitive elements				\\ \hline
\end{tabular}	
\end{setup}

\begin{rem}
 Note that while a weakly complete algebra is an algebra with additional structure (namely a topology), a weakly complete coalgebra is in general \emph{not} a coalgebra. This is due to the fact that the weakly complete comultiplication
 $
  \smfunc{\Delta}{C}{C\wcotimes C}
 $
 takes values in the completion $C\wcotimes C$, while for a coalgebra it would be necessary that it takes its values in $C\otimes C$ and the canonical inclusion map $C\otimes C \mapsto C\wcotimes C$ goes into the wrong direction (see also \cite[page 680]{MR2035107}). In particular, a Hopf algebra in the weakly complete category is not a Hopf algebra in general.
\end{rem}

Using the duality, we may translate theorems from the abstract category to the weakly complete category, for example the Fundamental Lemma of weakly complete algebras (Lemma \ref{lem: fundamental_lemma_of_weakly_complete_algebras}) follows directly from the the fundamental theorem of coalgebras, stating that every abstract coalgebra is the direct union of its finite-dimensional subcoalgebras. 
It should be mentioned that one of the first proofs of the fundamental theorem of coalgebras by Larson \cite[Prop. 2.5]{MR0209206} used this duality and worked in the framework of topological algebras to show the result about abstract coalgebras.


 Let $\Hopf$ be an abstract Hopf algebra and $H\coloneq \Hopf^*$ the corresponding weakly complete Hopf algebra. 
 Then the characters of $\Hopf$ are exactly the group like elements in $H$, while the infinitesimal characters of $\Hopf$ are exactly the primitive elements of $H$.
 This allows us to rephrase the scalar valued case of Theorem \ref{thm: Char:Lie}.

 \begin{thm}[Group like elements in a weakly complete Hopf algebra]
 Let $H$ be a densely graded weakly complete Hopf algebra over $\R$ or $\C$ with $H_0=\K$. Then the group like elements of $H$ form a closed Lie subgroup of the open unit group $H^\times$. The Lie algebra of this group is the weakly complete Lie algebra of primitive elements.
\end{thm}

\phantomsection
\addcontentsline{toc}{section}{References}
\bibliographystyle{new}
\bibliography{Hopf_lit}

\noindent{\small 
\textbf{Geir Bogfjellmo}, 
Institutt for matematiske fag, NTNU, 7491 Trondheim, Norway.\\ Email:
\href{mailto:geir.bogfjellmo@math.ntnu.no}{geir.bogfjellmo@math.ntnu.no}\\[2mm]
\textbf{Rafael Dahmen},
Fachbereich Mathematik, TU Darmstadt,
Schlo\ss{}gartenstr.\ 7, 64289 Darmstadt, Germany. Email:
\href{mailto:dahmen@mathematik.tu-darmstadt.de}{dahmen@mathematik.tu-darmstadt.de}\\[2mm]
\textbf{Alexander Schmeding},
Institutt for matematiske fag, NTNU, 7491 Trondheim, Norway. Email:
\href{mailto:alexander.schmeding@math.ntnu.no}{alexander.schmeding@math.ntnu.no}}\vfill
\end{appendix}

\end{document}